\documentclass{article}

\usepackage{a4wide}
\usepackage{amsfonts}
\usepackage{amssymb}
\usepackage{amsthm}
\usepackage{amsmath}
\usepackage{tikz}
\usepackage{multirow}

\newtheorem{theorem}{Theorem}[section]
\newtheorem{corollary}{Corollary}[section]
\theoremstyle{definition}
\newtheorem{definition}[theorem]{Definition}
\newtheorem{ex}{Example}[section]
\newtheorem{remark}{Remark}[section]
\newtheorem*{prob:P:T}{Problem}

\begin{document}

\title{A General Delta-Nabla Calculus of Variations on Time Scales
with Application to Economics\thanks{This is a preprint of a paper 
whose final and definite form will be published in the 
\emph{Int. J. Dyn. Syst. Differ. Equ.} (IJDSDE), ISSN 1752-3583. 
Paper submitted 17/Jul/2014; 
revised 21/Sept/2014 and 03/Oct/2014;
accepted for publication 05/Oct/2014.
Part of first author's Ph.D., carried 
out at the University of Aveiro under the
\emph{Doctoral Programme in Mathematics and Applications}
of Universities of Aveiro and Minho.}}

\author{Monika Dryl\\
\texttt{monikadryl@ua.pt}
\and Delfim F. M. Torres\\
\texttt{delfim@ua.pt}}

\date{Center for Research and Development in Mathematics and Applications (CIDMA)\\
Department of Mathematics, University of Aveiro, 3810--193 Aveiro, Portugal}

\maketitle


\begin{abstract}
We consider a general problem of the calculus of variations on time scales
with a cost functional that is the composition of a certain scalar function
with delta and nabla integrals of a vector valued field.
Euler--Lagrange delta-nabla differential equations are proved,
which lead to important insights in the process of discretization.
Application of the obtained results to a firm
that wants to program its production and investment policies
to reach a given production rate and to maximize its future
market competitiveness is discussed.

\bigskip


\bigskip

\end{abstract}


\section{Introduction}

The calculus of variations on time scales has been developing rapidly
in the past nine years, after the pioneering work \cite{MR2106410},
and is now a fertile area of research.
Indeed, in order to deal with nontraditional applications in economics,
where the system dynamics are described on a time scale partly continuous and partly discrete,
or to accommodate nonuniform sampled systems, one needs to work with variational problems
defined on a time scale \cite{MR2218315,Atici:comparison,MyID:267}.
Here we study general nonclassical problems of the calculus of variations
on time scales. More precisely, we consider the problem of minimizing
or maximizing a composition of delta and nabla integral functionals.
Main results include new necessary optimality conditions (Theorem~\ref{main})
that lead to better discretizations with relevance in economics.

The paper is organized as follows. In Section~\ref{sec:Preliminairies}
we collect the necessary background on the nabla and delta
calculus on time scales. In Section~\ref{sec:mainResults}
we formulate the general (nonclassical) mixed delta-nabla
problem \eqref{problem}--\eqref{punkty} of the calculus of variations on time scales.
We prove general necessary optimality conditions of Euler--Lagrange type in differential form
(Theorem~\ref{main}), which are then applied to the particular time scales $\mathbb{T}=\mathbb{R}$
(Corollary~\ref{cor in R}) and $\mathbb{T}=\mathbb{Z}$ (Corollary~\ref{cor in Z}).
In Section~\ref{sec:ApplEconomics} we consider an economic problem
describing a firm that wants to program its production and investment policies
to reach a given production rate and to maximize  its future market competitiveness.
The continuous case, denoted by $(P)$, was discussed in \cite{CastilloPedregal};
here we focus our attention on different discretizations of problem $(P)$,
in particular to the mixed delta-nabla discretizations that we call $(P_{\Delta\nabla})$ and $(P_{\nabla\Delta})$.
For these discrete problems the direct discretization of the Euler--Lagrange equation for $(P)$
does not lead to the solution of the problems: the results found by applying our Corollary~\ref{cor in Z}
to $(P_{\Delta\nabla})$ and $(P_{\nabla\Delta})$ are shown to be better. The comparison
is done in Section~\ref{sec:compare}. We end with Section~\ref{sec:conc} of conclusion
and future work.


\section{Preliminaries}
\label{sec:Preliminairies}

In this section we review some basic definitions and theorems
that are useful in the sequel. For more details concerning
the theory of time scales we refer to the books
\cite{BohnerDEOTS,MBbook2003}. For the calculus of variations
on time scales see \cite{GMT,TorresDeltaNabla,Martins:Torres}
and references therein. All the intervals in this paper
are time scale intervals.

\begin{definition}
A time scale $\mathbb{T}$ is an arbitrary nonempty closed subset of $\mathbb{R}$.
Given a time scale $\mathbb{T}$, the backward jump operator
$\rho:\mathbb{T} \rightarrow \mathbb{T}$ is defined
by $\rho(t):=\sup\lbrace s\in\mathbb{T}: s<t\rbrace$ for $t\neq \inf\mathbb{T}$
and $\rho(\inf\mathbb{T}) := \inf\mathbb{T}$ if $\inf\mathbb{T}>-\infty$.
The forward jump operator $\sigma:\mathbb{T}\rightarrow \mathbb{T}$
is defined by $\sigma(t):=\inf\lbrace s\in\mathbb{T}: s>t\rbrace$
for $t\neq \sup\mathbb{T}$ and $\sigma(\sup\mathbb{T}) := \sup\mathbb{T}$
if $\sup\mathbb{T}<+\infty$.
\end{definition}

A point $t\in\mathbb{T}$ is \emph{right-dense} or \emph{right-scattered},
\emph{left-dense} or \emph{left-scattered},
if $\sigma(t)=t$ or $\sigma(t)>t$, $\rho(t)=t$ or $\rho(t)<t$,
respectively.

\begin{definition}
The backward graininess function
$\nu:\mathbb{T} \rightarrow [0,\infty)$
is defined by $\nu(t):=t-\rho(t)$;
the forward graininess function
$\mu:\mathbb{T} \rightarrow [0,\infty)$
is defined by $\mu(t):=\sigma(t)-t$.
\end{definition}

\begin{ex}
If $\mathbb{T}=h\mathbb{Z}$, $h>0$, then $\sigma(t)=t+h$,
$\rho(t)=t-h$, and $\mu(t)=\nu(t) \equiv h$.
\end{ex}

To simplify the notation, we use $f^{\rho}(t):=f(\rho(t))$
and $f^{\sigma}(t):=f(\sigma(t))$. If $\mathbb{T}$ has
a right-scattered minimum $m$, then
we define $\mathbb{T}_{\kappa}:=\mathbb{T}-\lbrace m\rbrace$;
otherwise, we set $\mathbb{T}_{\kappa}:=\mathbb{T}$.
Similarly, if $\sup\mathbb{T}$ is finite and left-scattered, then
we define $\mathbb{T}^{\kappa}:=\mathbb{T}-\lbrace \sup\mathbb{T}\rbrace$;
otherwise, we set $\mathbb{T}^{\kappa}:=\mathbb{T}$.
Let us define the sets $\mathbb{T}^{\kappa^n}$,  $n\geq 2$, inductively:
$\mathbb{T}^{\kappa^1} :=\mathbb{T}^\kappa$ and
$\mathbb{T}^{\kappa^n} := (\mathbb{T}^{\kappa^{n-1}})^\kappa$,
$n\geq 2$. Similarly, $\mathbb{T}_{\kappa^1} := \mathbb{T}_\kappa$ and
$\mathbb{T}_{\kappa^n} := (\mathbb{T}_{\kappa^{n-1}})_\kappa$, $n\geq 2$.
Finally, we define $\mathbb{T}_{\kappa}^{\kappa} := \mathbb{T}_{\kappa} \cap \mathbb{T}^{\kappa}$.


\subsection{The nabla approach to time scales}
\label{sub:sec:2.1}

The nabla approach is based on the $\rho$ operator.

\begin{definition}[Section 3.1 of \cite{MBbook2003}]
We say that a function $f:\mathbb{T} \rightarrow \mathbb{R}$ is nabla differentiable
at $t\in\mathbb{T}_{\kappa}$ if there is a number $f^{\nabla}(t)$ such that for all
$\varepsilon >0$ there exists a neighborhood $U$ of $t$
(i.e., $U=(t-\delta, t+\delta)\cap\mathbb{T}$ for some $\delta>0$) such that
\begin{equation*}
|f^{\rho}(t)-f(s)-f^{\nabla}(t)(\rho(t)-s)| \leq
\varepsilon |\rho(t)-s| \mbox{ for all }  s\in U.
\end{equation*}
We say that $f^{\nabla}(t)$ is the nabla derivative of $f$ at $t$.
Moreover, $f$ is said to be nabla differentiable on $\mathbb{T}$ provided $f^{\nabla}(t)$
exists for all $t\in\mathbb{T}_{\kappa}$.
\end{definition}

\begin{theorem}[Theorem 8.39 of \cite{BohnerDEOTS}]
\label{rozniczka nabla}
Let $f:\mathbb{T} \rightarrow \mathbb{R}$ and $t\in\mathbb{T}_{\kappa}$.
If $f$ is continuous at $t$ and $t$ is left-scattered,
then $f$ is nabla differentiable at $t$ with
$$
f^{\nabla}(t)=\frac{f(t)-f(\rho(t))}{\nu(t)}.
$$
\end{theorem}

\begin{theorem}[Theorem 8.41. of \cite{BohnerDEOTS}]
\label{tw:differprop}
Let $f,g:\mathbb{T} \rightarrow \mathbb{R}$
be nabla differentiable at $t\in\mathbb{T_{\kappa}}$. Then,
\begin{enumerate}

\item the sum $f+g:\mathbb{T} \rightarrow \mathbb{R}$
is nabla differentiable at $t$ with
\begin{equation*}
(f+g)^{\nabla}(t)=f^{\nabla}(t)+g^{\nabla}(t);
\end{equation*}

\item for any constant $\alpha$,
$\alpha f:\mathbb{T} \rightarrow \mathbb{R}$
is nabla differentiable at $t$ with
\begin{equation*}
(\alpha f)^{\nabla}(t)=\alpha f^{\nabla}(t);
\end{equation*}

\item the product $fg:\mathbb{T} \rightarrow \mathbb{R}$
is nabla differentiable at $t$ with
\begin{equation*}
(fg)^{\nabla}(t)=f^{\nabla}(t)g(t)+f^{\rho}g^{\nabla}(t)=f(t)g^{\nabla}(t)+f^{\nabla}(t)g^{\rho}(t);
\end{equation*}

\item if $g(t)g^{\rho}(t)\neq 0$, then $f/g$ is nabla differentiable at $t$ with
\begin{equation*}
\bigg(\frac{f}{g}\bigg)^{\nabla}(t)=\frac{f^{\nabla}(t)g(t)-f(t)g^{\nabla}(t)}{g(t)g^{\rho}(t)}.
\end{equation*}
\end{enumerate}
\end{theorem}

\begin{definition}[Section 3.1 of \cite{MBbook2003}]
Let $\mathbb{T}$ be a time scale and $f:\mathbb{T}\rightarrow \mathbb{R}$.
We say that $f$ is ld-continuous if it is continuous at left-dense
points and its right-sided limits exists (finite) at all right-dense points.
\end{definition}

The set of all ld-continuous functions $f:\mathbb{T}\rightarrow \mathbb{R}$
is denoted by
\begin{equation*}
C_{ld}=C_{ld}(\mathbb{T})=C_{ld}(\mathbb{T},\mathbb{R})
\end{equation*}
and the set of all nabla differentiable functions
with ld-continuous derivative by
\begin{equation*}
C^{1}_{ld}=C^{1}_{ld}(\mathbb{T})=C^{1}_{ld}(\mathbb{T},\mathbb{R}).
\end{equation*}

\begin{theorem}[Theorems 8.46 and 8.47 of \cite{BohnerDEOTS} and Theorem 8 of \cite{MR2671876}]
If $a,b,c\in\mathbb{T}$, $a \leq c \leq b$, $\alpha\in\mathbb{R}$,
and $f,g\in C_{ld}\left(\mathbb{T}, \mathbb{R}\right)$, then:
\begin{enumerate}

\item $\int\limits_{a}^{b}(f(t)+g(t))\nabla t
=\int\limits_{a}^{b} f(t)\nabla t+\int\limits_{a}^{b}g(t)\nabla t$;

\item $\int\limits_{a}^{b}\alpha f(t)\nabla t
=\alpha \int\limits_{a}^{b} f(t)\nabla t$;

\item $\int\limits_{a}^{b} f(t)\nabla t
=\int\limits_{a}^{c} f(t)\nabla t+\int\limits_{c}^{b} f(t)\nabla t$;

\item $\int\limits_{a}^{a} f(t)\nabla t=0$;

\item if $f,g\in C_{ld}^1\left(\mathbb{T}, \mathbb{R}\right)$, then
$\int\limits_{a}^{b}f(t)g^{\nabla}(t)\nabla t
= \left.f(t)g(t)\right|^{t=b}_{t=a}-\int\limits_{a}^{b}f^{\nabla}(t)g(\rho(t))\nabla t$;

\item if $f(t) \geq 0$ for all $a<t\leq b$, then $\int\limits_{a}^{b}f(t)\nabla t\geq 0$;

\item if $t\in\mathbb{T}_{\kappa}$, then $\int_{\rho(t)}^{t}f(\tau)\nabla \tau=\nu(t)f(t)$.
\end{enumerate}
\end{theorem}


\subsection{The delta approach to time scales}
\label{sub:sec:2.2}

The delta calculus is similar to the nabla one (Section~\ref{sub:sec:2.1})
with $\sigma$ taking the role of operator $\rho$.

\begin{definition}[Section 1.1 of \cite{BohnerDEOTS}]
Let $f:\mathbb{T}\rightarrow\mathbb{R}$ and $t\in\mathbb{T}^{\kappa}$.
We define $f^{\Delta}(t)$ to be the number (provided it exists)
with the property that given any $\varepsilon >0$,
there is a neighborhood $U$ of $t$ such that
\begin{equation*}
\left|f^{\sigma}(t)-f(s)-f^{\Delta}(t)\left(\sigma(t)-s\right)\right|
\leq \varepsilon \left|\sigma(t)-s\right| \mbox{ for all }  s\in U.
\end{equation*}
We call $f^{\Delta}(t)$ the delta derivative of $f$ at $t$.
Function $f$ is delta differentiable on $\mathbb{T}^{\kappa}$ provided
$f^{\Delta}(t)$ exists for all $t\in\mathbb{T}^{\kappa}$. Then,
$f^{\Delta}:\mathbb{T}^{\kappa}\rightarrow\mathbb{R}$
is called the delta derivative of $f$ on $\mathbb{T}^{\kappa}$.
\end{definition}

\begin{theorem}[Theorem 1.16 of \cite{BohnerDEOTS}]
\label{rozniczka delta}
Let $f:\mathbb{T} \rightarrow \mathbb{R}$
and $t\in\mathbb{T}^{\kappa}$. If $f$ is continuous at $t$
and $t$ is right-scattered, then $f$ is delta differentiable at $t$ with
$$
f^{\Delta}(t)=\frac{f(\sigma(t))-f(t)}{\mu(t)}.
$$
\end{theorem}

\begin{theorem}[Theorem 1.20 of \cite{BohnerDEOTS}]
\label{tw:differpropdelta}
Let $f,g:\mathbb{T} \rightarrow \mathbb{R}$
be delta differentiable at $t\in\mathbb{T^{\kappa}}$. Then,
\begin{enumerate}

\item the sum $f+g:\mathbb{T} \rightarrow \mathbb{R}$ is delta differentiable at $t$ with
\begin{equation*}
(f+g)^{\Delta}(t)=f^{\Delta}(t)+g^{\Delta}(t);
\end{equation*}

\item for any constant $\alpha$, $\alpha f:\mathbb{T}\rightarrow\mathbb{R}$
is delta differentiable at $t$ with
\begin{equation*}
(\alpha f)^{\Delta}(t)=\alpha f^{\Delta}(t);
\end{equation*}

\item the product $fg:\mathbb{T}\rightarrow\mathbb{R}$
is delta differentiable at $t$ with
\begin{equation*}
(fg)^{\Delta}(t)=f^{\Delta}(t)g(t)+f^{\sigma}g^{\Delta}(t)
=f(t)g^{\Delta}(t)+f^{\Delta}(t)g^{\sigma}(t);
\end{equation*}

\item if $g(t)g^{\sigma}(t)\neq 0$, then $f/g$ is delta differentiable at $t$ with
\begin{equation*}
\left(\frac{f}{g}\right)^{\Delta}(t)
=\frac{f^{\Delta}(t)g(t)-f(t)g^{\Delta}(t)}{g(t)g^{\sigma}(t)}.
\end{equation*}
\end{enumerate}
\end{theorem}

\begin{definition}[Section 1.4 of \cite{MBbook2003}]
A function $f:\mathbb{T}\rightarrow \mathbb{R}$ is called rd-continuous
provided it is continuous at right-dense points in $\mathbb{T}$ and its
left-sided limits exist (finite) at all left-dense points in $\mathbb{T}$.
\end{definition}

The set of all rd-continuous functions
$f:\mathbb{T} \rightarrow \mathbb{R}$ is denoted by
$$
C_{rd}=C_{rd}(\mathbb{T})=C_{rd}(\mathbb{T},\mathbb{R}).
$$
The set of functions $f:\mathbb{T}\rightarrow \mathbb{R}$
that are delta differentiable and whose derivative
is rd-continuous is denoted by
$$
C^{1}_{rd}=C_{rd}^{1}(\mathbb{T})=C^{1}_{rd}(\mathbb{T},\mathbb{R}).
$$

\begin{theorem}[Theorems 1.75 and 1.77 of \cite{BohnerDEOTS}]
If $a,b,c\in\mathbb{T}$, $a\leq c\leq b$, $\alpha\in\mathbb{R}$,
and $f,g\in C_{rd}(\mathbb{T}, \mathbb{R})$, then
\begin{enumerate}

\item $\int\limits_{a}^{b}(f(t)+g(t))\Delta t
=\int\limits_{a}^{b} f(t)\Delta t+\int\limits_{a}^{b}g(t)\Delta t$;

\item $\int\limits_{a}^{b}\alpha f(t)\Delta t
=\alpha \int\limits_{a}^{b} f(t)\Delta t$;

\item $\int\limits_{a}^{b} f(t)\Delta t
=\int\limits_{a}^{c} f(t)\Delta t+\int\limits_{c}^{b} f(t)\Delta t$;

\item $\int\limits_{a}^{a} f(t)\Delta t=0$;

\item if $f,g\in C_{rd}^1(\mathbb{T}, \mathbb{R})$, then
$\int\limits_{a}^{b}f(t)g^{\Delta}(t)\Delta t =\left.f(t)g(t)\right|^{t=b}_{t=a}
-\int\limits_{a}^{b}f^{\Delta}(t)g(\sigma(t))\Delta t$;

\item if $f,g\in C_{rd}^1(\mathbb{T}, \mathbb{R})$, then
$\int\limits_{a}^{b}f(\sigma(t))g^{\Delta}(t)\Delta t
=\left.f(t)g(t)\right|^{t=b}_{t=a}
-\int\limits_{a}^{b}f^{\Delta}(t)g(t)\Delta t$;

\item if $f(t)\geq0$ for all $a\leq t < b$,
then $\int\limits_{a}^{b}f(t)\Delta t \geq 0$;

\item if $t\in\mathbb{T}^{\kappa}$, then
$\int\limits_{t}^{\sigma(t)}f(\tau)\Delta \tau=\mu(t)f(t)$.
\end{enumerate}
\end{theorem}


\subsection{Relation between delta and nabla approaches to time-scale calculus}
\label{sub:sec:relation:delta:nabla}

It is possible to relate the approach of Section~\ref{sub:sec:2.1}
with that of Section~\ref{sub:sec:2.2}.

\begin{theorem}[Theorems 2.5 and 2.6 of \cite{AticiGreen's_functions}]
\label{polaczenie rozniczka delta i nabla}
If $f:\mathbb{T}\rightarrow\mathbb{R}$ is delta differentiable on $\mathbb{T}^{\kappa}$
and $f^{\Delta}$ is continuous on $\mathbb{T}^{\kappa}$,
then $f$ is nabla differentiable on $\mathbb{T}_{\kappa}$ with
\begin{equation}
\label{polaczenie rozniczka nabla}
f^{\nabla}(t)=(f^{\Delta})^{\rho}(t) \textrm{ for all } t\in\mathbb{T}_{\kappa}.
\end{equation}
If $f:\mathbb{T}\rightarrow\mathbb{R}$ is nabla differentiable
on $\mathbb{T}_{\kappa}$ and $f^{\nabla}$ is continuous on $\mathbb{T}_{\kappa}$,
then $f$ is delta differentiable on $\mathbb{T}^{\kappa}$ with
\begin{equation}
\label{polaczenie rozniczka delta}
f^{\Delta}(t)=(f^{\nabla})^{\sigma}(t) \textrm{ for all } t\in\mathbb{T}^{\kappa}.
\end{equation}
\end{theorem}

\begin{theorem}[Proposition 7 of \cite{Gurses}]
\label{polaczenie calka delta i nabla}
If function $f:\mathbb{T}\rightarrow\mathbb{R}$ is continuous,
then for all $a,b\in\mathbb{T}$ with $a<b$ we have
\begin{gather}
\int\limits_{a}^{b}f(t)\Delta t
=\int\limits_{a}^{b}f^{\rho}(t)\nabla t,\label{polaczenie calka delta}\\
\int\limits_{a}^{b}f(t)\nabla t
=\int\limits_{a}^{b}f^{\sigma}(t)\Delta t. \label{polaczenie calka nabla}
\end{gather}
\end{theorem}

For a different approach relating the delta and the nabla calculi,
based on duality, we refer the reader to \cite{MR2771298,MR2957726}.


\section{Main results}
\label{sec:mainResults}

By $\mathcal{C}^{1}$ we denote the class of continuous
functions $y:[a,b]\rightarrow\mathbb{R}$ that are simultaneously
delta and nabla differentiable with $y^{\Delta}(t)$
and $y^{\nabla}(t)$ continuous on $[a,b]_{\kappa}^{\kappa}$.
Let $k,n \in \mathbb{N} = \{1, 2, \ldots\}$,
let $\mathbb{T}$ be a given time scale with at least three points,
and let $a,b\in\mathbb{T}$. We consider the following
general problem of the calculus of variations on time scales.

\begin{prob:P:T}
Find a function $y\in \mathcal{C}^{1}$ that extremizes (minimizes or maximizes) the functional
\begin{multline}
\label{problem}
\mathcal{L}[y]=H\left(\int\limits_{a}^{b}f_{1}(t,y^{\sigma}(t),y^{\Delta}(t))\Delta t,
\ldots, \int\limits_{a}^{b}f_{k}(t,y^{\sigma}(t),y^{\Delta}(t))\Delta t,\right.\\
\left.\int\limits_{a}^{b}f_{k+1}(t,y^{\rho}(t),y^{\nabla}(t))\nabla t, \ldots,
\int\limits_{a}^{b}f_{k+n}(t,y^{\rho}(t),y^{\nabla}(t)) \nabla t\right)
\end{multline}
subject to the boundary conditions
\begin{equation}
\label{punkty}
y(a)=y_{a}, \quad y(b)=y_{b},
\end{equation}
and under the assumptions that
\begin{enumerate}

\item function $H:\mathbb{R}^{n+k}\rightarrow\mathbb{R}$
has continuous partial derivatives with respect to its arguments,
which we denote by $H_{i}^{'}$, $i=1, \ldots, n+k$;

\item functions $(t,y,v) \rightarrow f_{i}(t,y,v)$ from
$[a,b]\times\mathbb{R}^{2}$ to $\mathbb{R}$, $i=1, \ldots, n+k$,
have continuous partial derivatives with respect to
$y$ and $v$ uniformly in $t\in [a,b]$, which we denote by $f_{iy}$ and $f_{iv}$, respectively;

\item functions $f_{i}$, $f_{iy}$, $f_{iv}$ are rd-continuous in $t\in [a,b]^{\kappa}$,
$i=1, \ldots, k$, and ld-continuous in $t\in [a,b]_{\kappa}$,
$i=k+1, \ldots, k+n$, for all $y\in \mathcal{C}^{1}$.
\end{enumerate}
\end{prob:P:T}

A function $y\in \mathcal{C}^{1}$ is said to be \emph{admissible} provided
it satisfies the boundary conditions \eqref{punkty}. In order to introduce
the notion of solution to our problem, we consider the following norm in $\mathcal{C}^{1}$:
$$
||y||_{1,\infty}:=||y^{\sigma}||_{\infty}
+||y^{\Delta}||_{\infty}+||y^{\rho}||_{\infty}+||y^{\nabla}||_{\infty},
$$
where $||y||_{\infty}:= \sup_{t \in [a,b] \cap \mathbb{T}_\kappa^\kappa} |y(t)|$.

\begin{definition}
\label{def:sol:P}
We say that an admissible function $\hat{y}$ is a local minimizer
(respectively, local maximizer) to problem \eqref{problem}--\eqref{punkty}
if there exists $\delta >0$ such that $\mathcal{L}[\hat{y}]\leq \mathcal{L}[y]$
(respectively, $\mathcal{L}[\hat{y}]\geq \mathcal{L}[y]$) for all admissible
functions $y\in \mathcal{C}^{1}$ satisfying the inequality $||y-\hat{y}||_{1,\infty}<\delta$.
\end{definition}

For brevity, in what follows we omit the argument of $H_{i}^{'}$. Precisely,
$H_{i}^{'}:=H_{i}^{'}(\mathcal{F}_{1}[\hat{y}], \ldots, \mathcal{F}_{k+n}[\hat{y}])$, where
$\mathcal{F}_{i}[\hat{y}] := \int\limits_{a}^{b}f_{i}(t,\hat{y}^{\sigma}(t),\hat{y}^{\Delta}(t))\Delta t$
for $i=1, \ldots, k$ and $\mathcal{F}_{i}[\hat{y}]
:= \int\limits_{a}^{b}f_{i}(t,\hat{y}^{\rho}(t),\hat{y}^{\nabla}(t))\nabla t$ for $i=k+1, \ldots, k+n$.
In contrast with \cite{MR3040923}, where integral conditions are investigated,
here we are interested in obtaining Euler--Lagrange type optimality
conditions in differential form.

\begin{remark}
If one considers the particular case where function $H$
in problem \eqref{problem}--\eqref{punkty}
does not depend on nabla operators, then one obtains exactly
the delta problem studied in \cite{MalinowskaTorresCompositionDelta}.
In this case, the assumptions we are considering for problem \eqref{problem}--\eqref{punkty}
coincide with the ones of \cite{MalinowskaTorresCompositionDelta}.
However, it should be noted that when it is written $\frac{\Delta}{\Delta t}$ or
$\frac{\nabla}{\nabla t}$ for some given expression, this is formal
and does not mean that one can really expand the delta (or nabla) derivative.
Such formal expressions are common in the literature of calculus of variations
(see, e.g., \cite[Theorem~1 of Section~4]{MR0160139},
\cite[Corollary~2 to Theorem~2.3]{MR1210325} or \cite[Section 6.1]{MR1363262}).
All our expressions are valid in integral form (see \cite{MR3040923}).
\end{remark}

\begin{theorem}[The delta-nabla Euler--Lagrange equations]
\label{main}
Let $\mathbb{\tilde{T}}$ be a time scale with $a, b \in \mathbb{\tilde{T}}$ and
$\mathbb{T} := [a,b] \cap \mathbb{\tilde{T}}$ having at least three points.
If $\hat{y}$ is a solution to problem \eqref{problem}--\eqref{punkty},
in the sense of Definition~\ref{def:sol:P},
then the following delta-nabla Euler--Lagrange equations
hold for all $t\in\mathbb{T}_{\kappa}^{\kappa}$:
\begin{multline}
\label{eq 1 delta}
\sum\limits_{i=1}^{k}H_{i}^{'}\cdot \left(
f_{iy}[\hat{y}](t)-\frac{\Delta}{\Delta t}f_{iv}[\hat{y}](t)
\right)
+\sum\limits_{i=k+1}^{k+n}H_{i}^{'}\cdot \left(
f_{iy}\{\hat{y}\}(\sigma(t))
-\frac{\Delta}{\Delta t}f_{iv}\{\hat{y}\}(t) \right)\\
+\frac{\Delta}{\Delta t}
\left[\sum\limits_{i=k+1}^{k+n}H_{i}^{'}\cdot
\nu(t)\cdot \left(
f_{iy}\{\hat{y}\}(t)-\frac{\nabla}{\nabla t}f_{iv}\{\hat{y}\}(t)\right)
\right]^{\sigma}(t)  = 0
\end{multline}
and
\begin{multline}
\label{eq 2 nabla}
\sum\limits_{i=1}^{k}H_{i}^{'}\cdot
\left(
f_{iy}[\hat{y}](\rho(t))-\frac{\nabla}{\nabla t}f_{iv}[\hat{y}](t)
\right)
+\sum\limits_{i=k+1}^{k+n}H_{i}^{'}\cdot
\left(
f_{iy}\{\hat{y}\}(t)-\frac{\nabla}{\nabla t}f_{iv}\{\hat{y}\}(t)
\right)\\
-\frac{\nabla}{\nabla t}
\left[
\sum\limits_{i=1}^{k}H_{i}^{'}\cdot
\mu(t)\cdot
\left(
f_{iy}[\hat{y}](t)-\frac{\Delta}{\Delta t}f_{iv}[\hat{y}](t)
\right) \right]^{\rho}(t) = 0,
\end{multline}
where
$[\hat{y}](t)=\left(t, \hat{y}^{\sigma}(t),\hat{y}^{\Delta}(t)\right)$ and
$\{\hat{y}\}(t)=\left(t,\hat{y}^{\rho}(t),\hat{y}^{\nabla}(t)\right)$.
\end{theorem}

\begin{proof}
Suppose that $\mathcal{L}\left[y\right]$ has a local extremum at $\hat{y}$.
Consider a variation $h\in \mathcal{C}^{1}$ of  $\hat{y}$ for which we define the function
$\phi:\mathbb{R} \rightarrow \mathbb{R}$ by
$\phi(\varepsilon)=\mathcal{L}\left[\hat{y}+\varepsilon h\right]$.
A necessary condition for $\hat{y}$ to be an extremizer for $\mathcal{L}\left[y\right]$
is given by $\phi'\left(\varepsilon\right)=0$ for $\varepsilon =0$.
Using the chain rule, we obtain that
\begin{multline*}
\phi'\left(0 \right)
=\sum\limits_{i=1}^{k}H_{i}^{'}\int\limits_{a}^{b}
\left(f_{iy}[\hat{y}](t)h^{\sigma}(t)
+f_{iv}[\hat{y}](t)h^{\Delta}(t)\right)\Delta t\\
+\sum\limits_{i=k+1}^{k+n}H_{i}^{'}\int\limits_{a}^{b}
\left(f_{iy}\{\hat{y}\}(t)h^{\rho}(t)
+f_{iv}\{\hat{y}\}(t)h^{\nabla}(t)\right)\nabla t = 0.
\end{multline*}
Using relations
\begin{equation*}
(fg)^{\nabla}(t)=f^{\nabla}(t)g(t)+f^{\rho}g^{\nabla}(t)
=f(t)g^{\nabla}(t)+f^{\nabla}(t)g^{\rho}(t)
\end{equation*}
and
\begin{equation*}
(fg)^{\Delta}(t)=f^{\Delta}(t)g(t)+f^{\sigma}g^{\Delta}(t)
=f(t)g^{\Delta}(t)+f^{\Delta}(t)g^{\sigma}(t),
\end{equation*}
one has
\begin{equation*}
\left(
f_{iv}[\hat{y}](t)h(t)\right)^{\Delta}
=f_{iv}[\hat{y}](t)h^{\Delta}(t)+
\left(f_{iv}[\hat{y}](t)\right)^{\Delta}h^{\sigma}(t)
\end{equation*}
and
\begin{equation*}
\left(
f_{iv}\{\hat{y}\}(t)h(t)\right)^{\nabla}
=f_{iv}\{\hat{y}\}(t)h^{\nabla}(t)+
\left(f_{iv}\{\hat{y}\}(t)\right)^{\nabla}h^{\rho}(t).
\end{equation*}
Integrating both sides from $t=a$ to $t=b$ and having in mind that from
\eqref{punkty} one has $h(a) = h(b) = 0$, we obtain that
\begin{multline*}
\int\limits_{a}^{b}\sum\limits_{i=1}^{k}H_{i}^{'}\cdot
\left(f_{iy}[\hat{y}](t)
-\left(f_{iv}[\hat{y}](t)\right)^{\Delta}\right)h^{\sigma}(t)\Delta t\\
+\int\limits_{a}^{b}\sum\limits_{i=k+1}^{k+n}H_{i}^{'}\cdot
\left(f_{iy}\{\hat{y}\}(t)
-\left(f_{iv}\{\hat{y}\}(t)\right)^{\nabla}\right)h^{\rho}(t)\nabla t = 0.
\end{multline*}
Let us denote
\begin{equation*}
\begin{split}
s(t) &:= \sum\limits_{i=1}^{k}H_{i}^{'}\cdot
\left(f_{iy}[\hat{y}](t)
-\left(f_{iv}[\hat{y}](t)\right)^{\Delta}\right),\\
r(t) &:= \sum\limits_{i=k+1}^{k+n}H_{i}^{'}\cdot
\left(f_{iy}\{\hat{y}\}(t)-\left(f_{iv}\{\hat{y}\}(t)\right)^{\nabla}\right).
\end{split}
\end{equation*}
Then,
\begin{equation*}
\int\limits_{a}^{b}s(t)h^{\sigma}(t)\Delta t
+\int\limits_{a}^{b}r(t)h^{\rho}(t)\nabla t = 0.
\end{equation*}
Now we split the proof into two cases. First we use \eqref{polaczenie calka nabla}
of Theorem~\ref{polaczenie calka delta i nabla} and \eqref{polaczenie rozniczka delta}
of Theorem~\ref{polaczenie rozniczka delta i nabla} to obtain the Euler--Lagrange
equation \eqref{eq 1 delta}. Next we apply \eqref{polaczenie calka delta} of
Theorem~\ref{polaczenie calka delta i nabla} and \eqref{polaczenie rozniczka nabla}
of Theorem~\ref{polaczenie rozniczka delta i nabla} to receive
the latter Euler--Lagrange equation \eqref{eq 2 nabla}.

(i) Since $h$ is nabla differentiable, we have that
$h^{\rho}(t)=h(t)-\nu(t) h^{\nabla}(t)$
(cf. item (iv) of \cite[Theorem~3.2]{da:jb:le:ap:hnt}) and thus
$$
\int\limits_{a}^{b}s(t)h^{\sigma}(t)\Delta t
+\int\limits_{a}^{b}\left(r(t)h(t)-r(t)\nu (t)h^{\nabla}(t)\right)\nabla t = 0.
$$
Using equation \eqref{polaczenie calka nabla}
of Theorem~\ref{polaczenie calka delta i nabla}, it follows that
$$
\int\limits_{a}^{b}s(t)h^{\sigma}(t)\Delta t
+\int\limits_{a}^{b}\left[(rh)^{\sigma}(t)
-(r\nu)^{\sigma}(t)(h^{\nabla})^{\sigma}(t)\right]\Delta t = 0.
$$
Therefore, from equation \eqref{polaczenie rozniczka delta}
of Theorem~\ref{polaczenie rozniczka delta i nabla}, we obtain
$$
\int\limits_{a}^{b}s(t)h^{\sigma}(t)\Delta t
+\int\limits_{a}^{b}\left[(rh)^{\sigma}(t)-(r\nu)^{\sigma}(t)h^{\Delta}(t)\right]\Delta t = 0.
$$
Integrating the second part of the latter integral gives
$$
\int\limits_{a}^{b}(r\nu)^{\sigma}(t)h^{\Delta}(t)\Delta t
=(r\nu)^{\sigma}(t)h(t)\bigg|_{a}^{b}
-\int\limits_{a}^{b}h^{\sigma}(t)\frac{\Delta}{\Delta t}(r\nu)^{\sigma}(t)\Delta t
$$
and it follows that
$$
\int\limits_{a}^{b}\left[s(t)h^{\sigma}(t)
+r^{\sigma}(t)h^{\sigma}(t)+h^{\sigma}(t)\frac{\Delta}{\Delta t}(r\nu)^{\sigma}(t)\right]\Delta t = 0.
$$
Thus,
$$
\int\limits_{a}^{b}\left[s(t)
+r^{\sigma}(t)+\frac{\Delta}{\Delta t}(r\nu)^{\sigma}(t)\right]h^{\sigma}(t)\Delta t = 0.
$$
From the fundamental lemma of the delta calculus of variations (cf. Lemma~8 of \cite{MR1908827}
and Lemma~3.2 of \cite{MR2405376}), we get the Euler--Lagrange equation
$$
s(t)+r^{\sigma}(t)+\frac{\Delta}{\Delta t}(r\nu)^{\sigma}(t) = 0
$$
and therefore \eqref{eq 1 delta} holds.

(ii) Since $h$ is delta differentiable, the following relation holds
(cf. item (iv) of \cite[Theorem~1.3]{mb:gg:ap}):
$$
h^{\sigma}(t)=h(t)+\mu(t) h^{\Delta}(t).
$$
We then obtain that
$$
\int\limits_{a}^{b}s(t)h(t)+s(t)\mu (t) h^{\Delta}(t)\Delta t
+\int\limits_{a}^{b}r(t)h^{\rho}(t)\nabla t = 0.
$$
Using equation \eqref{polaczenie calka delta} of
Theorem~\ref{polaczenie calka delta i nabla},
$$
\int\limits_{a}^{b}\left[s^{\rho}(t)h^{\rho}(t)
+\left(s\mu\right)^{\rho}(t) (h^{\Delta})^{\rho}(t)
+r(t)h^{\rho}(t)
\right]\nabla t = 0.
$$
It follows, from equation \eqref{polaczenie rozniczka nabla}
of Theorem~\ref{polaczenie rozniczka delta i nabla}, that
$$
\int\limits_{a}^{b}\left[s^{\rho}(t)h^{\rho}(t)
+\left(s\mu\right)^{\rho}(t) h^{\nabla}(t)
+r(t)h^{\rho}(t)
\right]\nabla t = 0.
$$
Integrating the second item of the above integral,
$$
\int\limits_{a}^{b}
(s\mu )^{\rho}(t) h^{\nabla}(t)\nabla t=
(s\mu)^{\rho}(t) h(t)
\bigg|_{a}^{b}-\int\limits_{a}^{b}
\frac{\nabla}{\nabla t}(s\mu)^{\rho}(t) h^{\rho}(t)\nabla t,
$$
we obtain
$$
\int\limits_{a}^{b}\left[s^{\rho}(t)h^{\rho}(t)
+r(t)h^{\rho}(t)- h^{\rho}(t)
\frac{\nabla}{\nabla t}(s\mu )^{\rho}(t)\right]\nabla t = 0
$$
and then
$$
\int\limits_{a}^{b}\left[s^{\rho}(t)
+ r(t)- \frac{\nabla}{\nabla t}(s\mu)^{\rho}(t)\right]h^{\rho}(t)\nabla t = 0.
$$
From the fundamental lemma of the nabla calculus of variations
(cf. Lemma~15 of \cite{MR2671876}), we get the Euler--Lagrange equation
$$
s^{\rho}(t)+ r(t)- \frac{\nabla}{\nabla t}(s\mu)^{\rho}(t) = 0
$$
and therefore \eqref{eq 2 nabla} holds.
\end{proof}

\begin{corollary}[Euler--Lagrange equation (3.17) of \cite{CastilloPedregal}]
\label{cor in R}
Let $a, b \in \mathbb{R}$ with $a < b$.
If $y$ is solution to problem
\begin{gather*}
\mathcal{L}[y]=H\left(\int\limits_{a}^{b}f_{1}(t,y(t),y'(t)) dt,
\int\limits_{a}^{b} f_{2}(t,y(t),y'(t))dt\right)\longrightarrow \text{extr}\\
y(a)=y_{a}, \quad y(b)=y_{b},
\end{gather*}
then the following differential equation holds:
\begin{multline}
\label{E-L continuous}
H_{1}^{'}(F_1,F_2) \cdot \left(
\frac{\partial f_{1}}{\partial y}(t,y(t),y'(t))-\frac{d}{d t}
\frac{\partial f_{1}}{\partial y'}(t,y(t),y'(t))\right)\\
+ H'_{2}(F_1,F_2) \cdot \left(
\frac{\partial f_{2}}{\partial y}(t,y(t),y'(t))-\frac{d}{d t}
\frac{\partial f_{2}}{\partial y'}(t,y(t),y'(t))\right)=0
\end{multline}
for all $t \in [a,b]$, where
$$
F_i = \int\limits_{a}^{b} f_{i}(t,y(t),y'(t)) dt, \quad i = 1, 2.
$$
\end{corollary}

\begin{proof}
Let $\mathbb{\tilde{T}}=\mathbb{R}$ and $k=n=1$.
The result follows from Theorem~\ref{main}.
\end{proof}

Let $\varphi : \mathbb{R}^3 \rightarrow \mathbb{R}$.
In what follows we use $\partial_i \varphi$, $i \in \{1, 2, 3\}$,
to denote the partial derivative of $\varphi$
with respect to its $i$th argument.

\begin{corollary}
\label{cor in Z}
Let $a, b \in \mathbb{N}$ with $b - a > 1$
and denote by $\Delta y(t)$ and $\nabla y(t)$ the standard
forward and backward difference operators, that is,
$\Delta y(t) := y(t+1) - y(t)$
and $\nabla y(t) := y(t) - y(t-1)$.
If $y$ is solution to problem
\begin{gather*}
\mathcal{L}[y]=H\left(\sum\limits_{t=a}^{b-1} f(t,y(t+1),\Delta y(t)),
\sum\limits_{t=a+1}^{b} g(t,y(t-1),\nabla y(t))\right)\longrightarrow \text{extr}\\
y(a)=y_{a}, \quad y(b)=y_{b},
\end{gather*}
then both Euler--Lagrange difference equations
\begin{multline}
\label{eq 1 delta Z}
H_{1}^{'}(F,G) \cdot \left[
\partial_2 f(t,y(t+1),\Delta y(t)) -\Delta
\partial_3 f(t,y(t+1),\Delta y(t))\right]\\
+ H'_{2}(F,G) \cdot \left[
\partial_2 g(t+1,y(t),\nabla y(t+1))-\Delta
\partial_3 g(t,y(t-1),\nabla y(t))\right]\\
+H'_{2}(F,G) \cdot \Delta \left[
\partial_2 g(t+1,y(t),\nabla y(t+1))
-\nabla \partial_3 g(t+1,y(t),\nabla y(t+1))\right] = 0
\end{multline}
and
\begin{multline}
\label{eq 2 nabla Z}
H_{1}^{'}(F,G) \cdot \left[
\partial_2 f(t-1,y(t),\Delta y(t-1))-\nabla
\partial_3 f(t,y(t+1),\Delta y(t))\right]\\
- H_{1}^{'}(F,G) \cdot \nabla \left[
\partial_2 f(t-1,y(t),\Delta y(t-1))
-\Delta \partial_3 f(t-1,y(t),\Delta y(t-1))\right]\\
+ H_{2}^{'}(F,G) \cdot \left[
\partial_2 g(t,y(t-1),\nabla y(t))-\nabla
\partial_3 g(t,y(t-1),\nabla y(t))\right] = 0
\end{multline}
hold for $t \in \{a+1, \ldots, b-1\}$, where
$$
F := \sum\limits_{t=a}^{b-1} f(t,y(t+1),\Delta y(t)), \quad
G := \sum\limits_{t=a+1}^{b} g(t,y(t-1),\nabla y(t)).
$$
\end{corollary}

\begin{proof}
The result is a direct consequence of Theorem~\ref{main}
with $\mathbb{\tilde{T}}=\mathbb{Z}$ and $k=n=1$.
\end{proof}


\section{Application to Economics}
\label{sec:ApplEconomics}

In this section we introduce an economic problem
that is considered in continuous (Example~\ref{ex:1})
and discrete (Example~\ref{ex:2}) cases.
We consider a firm that wants to program
its production and investment policies in order to gain a desirable production
level and maximize its market competitiveness. Our idea is to discretize
necessary optimality conditions of Euler--Lagrange type ($EL_P$)
and the (continuous) problem $P$ in different ways,
combining forward ($\Delta$) and backward ($\nabla$) discretization operators into a mixed operator $D$.
One can apply the variational principle to problem $P$ obtaining the respective
Euler--Lagrange equation $EL_P$ (Corollary~\ref{cor in R}), and then discretize it using $D$,
obtaining $(EL_P)_D$; or we can begin by discretizing problem $P$ into $P_D$ and then
develop the respective variational principle, obtaining $EL_{P_D}$ (Theorem~\ref{main}).
This is illustrated in Figure~\ref{diagram}.
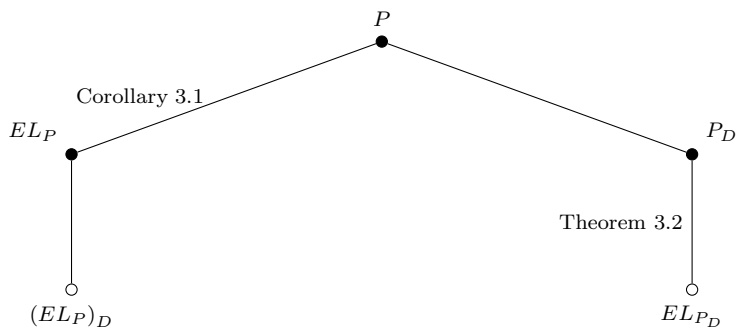
\begin{figure}[htbp]
\begin{center}
\begin{tikzpicture}[scale=1.5,font=\footnotesize]
\tikzstyle{solid node}=[circle,draw,inner sep=1.5,fill=black]
\tikzstyle{hollow node}=[circle,draw,inner sep=1.5]
\tikzstyle{level 1}=[level distance=10mm,sibling distance=5.5cm]
\tikzstyle{level 2}=[level distance=12mm,sibling distance=1.5cm]
\tikzstyle{level 3}=[level distance=12mm,sibling distance=2cm]
\node(0)[solid node,label=above:{$P$}]{}
child{node[solid node,label=above left:{$EL_P$}]{}
child{node[hollow node,label=below:{$\left(EL_P\right)_{D}$}]{} edge from parent node[left]{}}
edge from parent node[left,xshift=-5]{Corollary~\ref{cor in R}}
}
child{node[solid node,label=above right:{$P_D$}]{}
child{node[hollow node,label=below:{$EL_{P_D}$}]{} edge from parent node[left]{Theorem~\ref{main}}}
edge from parent node[right,xshift=5]{}
};
\end{tikzpicture}
\end{center}
\caption{\label{diagram} Diagram illustrating different
discretizations for a variational problem $P$.}
\end{figure}
Note that, in general, $\left(EL_P\right)_{D}$ is different from $EL_{P_D}$.
Four different problems $P_D$, four Euler--Lagrange equations $EL_{P_D}$
and four Euler--Lagrange equations $(EL_{P})_D$
are discussed and investigated. In what follows,
$$
\Delta y(t) := y^{\sigma}(t) - y(t),
\quad
\nabla y(t) := y(t) - y^{\rho}(t).
$$
In particular, if $\mathbb{T}$ has a maximum $M$, then
$\Delta y(M) = 0$; if $\mathbb{T}$ has a minimum $m$, then
$\nabla y(m) = 0$.


\subsection{Direct discretizations of the continuous Euler--Lagrange equation}
\label{sec:AP:1}

The next example is borrowed from Section~6 of \cite{CastilloPedregal}.

\begin{ex}[A continuous problem of the calculus of variations \cite{CastilloPedregal}]
\label{ex:1}
We consider a firm trying to program its production and investment policies in order
to reach a desirable production level and to maximize its future market competitiveness
at time horizon $T$. The firm competitiveness is measured by the function $f(k(T),a(T))$,
which depends on the accumulated capital $k(T)$ (accumulated goods devoted to production)
and accumulated technology $a(T)$ (capability given by the practical application
of knowledge and experience), both at time horizon $T$. We assume that the function
measuring the firm market competitiveness is the product of the accumulated capital
with the accumulated technology, that is,
\begin{equation}
\label{eq:two_components}
f(k(T),a(T))=k(T)^{\gamma_{1}}a(T)^{\gamma_{2}},
\end{equation}
where $\gamma_{1}$ and $\gamma_{2}$ are constants measuring the absolute and relative importance/influence
of capital and technology competitiveness, respectively. More precisely, the firm may decide to sell
products at a small or no benefit, or even losses, if due to this decision the firm can gain experience
and get technology acquisition. Firm's measure of competitiveness is chosen as the product
of two components because there is a strict dependence between capital
and technology. Indeed, one affects the other, and huge differences between them cannot be allowed.
This means that a lack of one of those components must be compensated by a large amount of the other
in order to reach the same competitiveness level.
The firm starts operating at time $t=0$ and accumulates capital over time as
$$
k(T)=\int\limits_{0}^{T} e^{-\rho(T-t)}\left[y(t)p(t)-c(y(t),y'(t)) \right] dt,
$$
where $\rho$ is the discount rate, $p(t)$ is the unit product price at time $t$,
and $y(t)$ is the actual production rate at time $t$. The accumulated capital
depends also on function $c\left(y(t),y'(t)\right)$, which is
the cost of producing $y(t)$ units of product at time $t$ plus technology increases.
In our model the cost function has the following form:
$$
c(y(t),y'(t))=c_{0}+c_{1}y(t)+c_{2}y'^{2}(t),
$$
where $c_{0}$, $c_{1}$, $c_{2}$ are positive constants.
The second component of \eqref{eq:two_components}
is the accumulated technology, which is the discounted integral
of the technology acquisition rate over time:
$$
a(T)=\int\limits_{0}^{T} e^{-\rho(T-t)}g(y(t),y'(t))dt
$$
with
$$
g(y(t),y'(t))=\lambda y(t)+\beta \sqrt{y'(t)+b},
$$
where $\lambda$, $\beta$ and $b$ are positive constants.
Function $g$ describes the acquisition technology rate at time $t$.
It depends on the actual sales rate $y(t)$ (equal to the actual production
rate at the same time) and $y'(t)$, the actual production rate change. The $y(t)$ argument accounts
for machines, other technology components, gained experience, etc., while the $y'(t)$ argument accounts
for technology due to changes on sales rate. This means that large positive or negative changes on sales
rate $y'(t)$ forces the firm to make decisions about technology increases: it may be a starting
point for the increase of production or a warning to avoid decrease.
All constants used in the model are positive and have a precise interpretation.
It is also worth to emphasise that both the production cost function $c$ and the acquisition technology
function $g$ depend on argument $y'$. However, in the first function, argument $y'$ is
of higher-order than in the latter, motivated by the fact that incorporation of technology
into a production process is very difficult, generates costs, and requires time to be checked.
The sales relationship is given by
$$
h(y(t),p(t))=(y(t)-y_{0})(p(t)-p_{0})-B=0.
$$
It has this hyperbolic form in order to express the assumption
that sales increase when the unit price decrease. Moreover,
it also gives the lower limit for the sales $(y_{0})$
and the lower limit for the unit price $(p_{0})$.
There is an upper bound $b$ for the size of production
rate change so that $|y'|\leq b$.
The economic problem under consideration is
$$
\max\limits_{y} f(k(T),a(T)) = \min\limits_{y} \left[ -k(T)a(T)\right]
= \min\limits_{y} K(T)a(T),
$$
where $K(T)=-k(T)$. More precisely, we consider $\gamma_{1}=\gamma_{2}=1$,
and the problem takes the form
$$
\min\limits_{y} \left(
\int\limits_{0}^{T} e^{-\rho(T-t)}\left[c(y(t),y'(t))-y(t)p(t)\right]dt
\right)
\left(
\int\limits_{0}^{T} e^{-\rho(T-t)}g(y(t),y'(t))dt
\right)
$$
subject to given boundary conditions
$$
y(0)=y_{0}, \quad y(T)=y_{T},
$$
where $y_{0}$ is the initial sales rate and $y_{T}$
is the target sales rate at time $t=T$.
This problem is denoted in the sequel by $(P)$.
Note that here $\rho$ is the discount rate (not to be confused with the backward jump operator $\rho(t)$
of time scales, which in the discrete case, to be considered in this section, is always expressed by the index $k-1$).
For problem $(P)$ the Euler--Lagrange differential equation \eqref{E-L continuous} takes the form
\begin{multline}
\label{eq:contELeq}
a(T) \, e^{-\rho(T-t)}\left[c_{1}-p(t)-2c_{2}(\rho y'(t)+y''(t))\right]\\
+K(T) \, e^{-\rho(T-t)}\left[\lambda-\frac{\beta}{2}\left(\frac{\rho}{\sqrt{y'(t)+b}}
-\frac{y''(t)}{2\sqrt{(y'(t)+b)^{3}}}\right)\right]=0.
\end{multline}
The solution of the continuous problem $(P)$ is found by solving
the Euler--Lagrange equation \eqref{eq:contELeq}. It turns out that this is a highly nonlinear
differential equation of second order, for which no analytical solution is known.
In other words, to solve the continuous problem one needs to apply a suitable discretization.
This is exactly one of the main motivations of our study: to provide an appropriate theory of discretization.
\end{ex}

A discretization can always be done in two different ways:
using the delta or the nabla approach. In the next example
we consider four different discretizations for the problem $(P)$
of Example~\ref{ex:1} and the corresponding four discretizations
of the Euler--Lagrange equation \eqref{eq:contELeq}.

\begin{ex}
\label{ex:2}
Consider a firm that wants to program its production and investment policies
to reach a given production rate $k(T)$, $T \in \mathbb{N}$,
and to maximize its future market competitiveness
at time horizon $T$. Economic models, leading to the maximization
of a variational functional, are presented below
and are based on the following assumptions:
\begin{enumerate}

\item
The firm competitiveness is measured by the function $f(k(T),a(T))$,
which depends on the accumulated capital $k(T)$ and on the accumulated technology $a(T)$
both at time horizon $T$. Here, the function to measure
the firm market competitiveness is assumed to be of form
\begin{equation}
\label{function}
f(k(T),a(T))=k(T)^{\gamma_{1}}a(T)^{\gamma_{2}}
\end{equation}
with $\gamma_{1}$ and $\gamma_{2}$ given constants that measure the absolute
and relative importance of capital and technology competitiveness, respectively.

\item
The acquisition technology rate is given by function
$g\left(y(t_{k+1}),\Delta y(t_{k})\right)$ (delta version) or
$g\left(y(t_{k-1}),\nabla y(t_{k})\right)$ (nabla version),
where $y(s)$ is the sales rate at time $s$,
which we assume equal to the actual production rate at the same point of time,
that is, $\Delta y(t_{k})$ (delta version) or $\nabla y(t_{k})$ (nabla version)
are the actual production rate change at time $t_{k}$.

\item
The firm starts operating at point $t_{0}=0$ and accumulates capital as
\begin{equation}
\label{capital delta}
K_{\Delta}(T)=
\sum\limits_{t_{k}=0}^{T-1}
(1+\rho)^{t_{k}-T}\left(c_{0}+c_{1}y_{k+1}+c_{2}\left(\Delta {y}_{k}\right)^{2}-y_{k+1}p_{k+1}\right)
\end{equation}
(delta version) or
\begin{equation}
\label{capital nabla}
K_{\nabla}(T)=
\sum\limits_{t_{k}=1}^{T}
(1-\rho)^{T-t_{k}}\left(c_{0}+c_{1}y_{k-1}+c_{2}\left(\nabla {y}_{k}\right)^{2}-y_{k-1}p_{k-1}\right)
\end{equation}
(nabla version), where $\rho$ is the discount rate,
$p_{k}=p(t_{k})$ is the unit product price,
$y_{k}=y(t_{k})$ is the sales rate at time $t_{k}$, and
$c(y_{k+1},\Delta y_{k})$ (delta) or $c(y_{k-1},\nabla y_{k})$ (nabla)
is the cost of producing $y_{k+1}$ (delta) or $y_{k-1}$ (nabla) units of product
at time $t_{k+1}$ (delta) or $t_{k-1}$ (nabla) plus technology increases.

\item
The accumulate technology is given by
\begin{equation}
\label{technology delta}
a_{\Delta}(T)=
\sum\limits_{t_{k}=0}^{T-1}
(1+\rho)^{t_{k}-T}
\left(\lambda y_{k+1}+\beta\sqrt{\Delta {y}_{k}+b}\right)
\end{equation}
(delta version) or
\begin{equation}
\label{technology}
a_{\nabla}(T)=
\sum\limits_{t_{k}=1}^{T}
(1-\rho)^{T-t_{k}}
\left(\lambda y_{k-1}+\beta\sqrt{\nabla {y}_{k}+b}\right)
\end{equation}
(nabla version).

\item
The price-sales relationship regulating the market is given by the equation
\begin{equation}
\label{price-sales delta}
h(y_{k+1},p_{k+1})=(y_{k+1}-y_{0})(p_{k+1}-p_{0})-B=0
\end{equation}
(delta version) or by the equation
\begin{equation}
\label{price-sales nabla}
h(y_{k-1},p_{k-1})=(y_{k-1}-y_{0})(p_{k-1}-p_{0})-B=0
\end{equation}
(nabla version).
There is an upper bound $b$ for the size of production rate change,
so that $\left|\Delta y_{k}\right|\leq b$ (delta) or
$\left|\nabla y_{k}\right|\leq b$ (nabla).

\item
Two boundary conditions are given:
\begin{equation}
\label{conditions}
y(0)= y_0,\quad y(T)= y_T,
\end{equation}
which are the initial sales rate at point $t_{0}=0$
and the target sales rate at the terminal point of time $t_{k}=T$.
\end{enumerate}
Then, the firm problem is stated as follows:
\begin{equation*}
\max\limits_{y_{k}}k(T)^{\gamma_{1}}a(T)^{\gamma_{2}}
\end{equation*}
subject to the hypotheses \eqref{function}--\eqref{conditions}.
For illustrative purposes, and to be coherent with Example~\ref{ex:1}
borrowed from \cite{CastilloPedregal},
we assume $\gamma_{1}=\gamma_{2}=1$
and we transform the maximization problem into
an equivalent minimization process:
\begin{equation*}
\min\limits_{y_{k}}(-k(T))a(T)=\min\limits_{y_{k}}K(T)a(T).
\end{equation*}
Each component of the objective functional $f(K(T),a(T))$ may be discretized
in two ways (using the delta or the nabla approach).
Due to this reason, we obtain four different
discrete problems of the calculus of variations:
\begin{enumerate}

\item Problem $(P_{\Delta\nabla})$ with cost functional
$\min\limits_{y_{k}} K_{\Delta}(T)a_{\nabla}(T)$;

\item Problem $(P_{\nabla\Delta})$ with cost functional
$\min\limits_{y_{k}} K_{\nabla}(T)a_{\Delta}(T)$;

\item Problem $(P_{\Delta\Delta})$ with cost functional
$\min\limits_{y_{k}} K_{\Delta}(T)a_{\Delta}(T)$;

\item Problem $(P_{\nabla\nabla})$ with cost functional
$\min\limits_{y_{k}} K_{\nabla}(T)a_{\nabla}(T)$;
\end{enumerate}
where $K_{\mathcal{D}}(T)$ and  $a_{\mathcal{D}}(T)$, $\mathcal{D}\in\left\{\Delta,\nabla\right\}$,
are defined as in \eqref{capital delta}--\eqref{technology}.
With the notation of Section~\ref{sec:mainResults},
such functionals consist of the following integrands:
\begin{equation*}
\begin{split}
f_{1\Delta}&=(1+\rho)^{t_{k}-T}(c_{0}+c_{1}y_{k+1}+c_{2}\left(\Delta {y}_{k}\right)^{2}-y_{k+1}p_{k+1}),\\
f_{1\nabla}&=(1-\rho)^{T-t_{k}}(c_{0}+c_{1}y_{k-1}+c_{2}\left(\nabla {y}_{k}\right)^{2}-y_{k-1}p_{k-1}),\\
f_{2\Delta}&= (1+\rho)^{t_{k}-T}
\left(\lambda y_{k+1}+\beta\sqrt{\Delta {y}_{k}+b}\right),\\
f_{2\nabla}&=(1-\rho)^{T-t_{k}}
\left(\lambda y_{k-1}+\beta\sqrt{\nabla {y}_{k}+b}\right),
\end{split}
\end{equation*}
where
$f_{i\Delta}=f_{i\Delta}(t_{k}, y_{k+1}, \Delta y_{k})$,
$f_{i\nabla}=f_{i\nabla}(t_{k}, y_{k-1}, \nabla y_{k})$,
$i=1,2$, and
function $f_{1\mathcal{D}}$ is associated with functional $K_{\mathcal{D}}(T)$
and function $f_{2\mathcal{D}}$ is associated with
functional $a_{\mathcal{D}}(T)$, $\mathcal{D} \in \left\{\Delta,\nabla\right\}$.
Using the same discretization as the one from $(P)$ to $(P_{\Delta\nabla})$,
the Euler--Lagrange equation \eqref{E-L continuous} is discretized into
\begin{equation}
\label{E-L delta nabla:g}
a_{\nabla}(T)\cdot\left(\frac{\partial f_{1\Delta}}{\partial y_{k+1}}
-\Delta\frac{\partial f_{1\Delta}}{\partial \Delta y_{k}}\right)
+K_{\Delta}(T)\cdot\left(\frac{\partial f_{2\nabla}}{\partial y_{k-1}}
-\nabla\frac{\partial f_{2\nabla}}{\partial \nabla y_{k}}\right)=0,
\end{equation}
which for our economic problem $(P)$ takes the form
\begin{multline}
\label{eq:EL:P:after7.17}
\tag{$(EL_P)_{\Delta\nabla}$}
a_{\nabla}(T)(1+\rho)^{t_{k}-T}\left[
c_{1}-p_{0}+\frac{By_{0}}{(y_{k+1}-y_{0})^{2}}
-2c_{2}\left(\rho\Delta y_{k}+(1+\rho)\Delta^{2} y_{k}\right)\right]\\
+ K_{\Delta}(T)(1-\rho)^{T-t_{k}}\left[
\lambda - \frac{
\beta \left(
\rho\sqrt{\nabla {y}_{k}+b}-\nabla\sqrt{\nabla {y}_{k}+b}\right)}
{2\sqrt{\nabla {y}_{k}+b}\sqrt{\nabla {y}_{k-1}+b}}\right] = 0,
\end{multline}
valid for $t_{k}\in\mathbb{T}_{\kappa}^{\kappa}$. Note that
we start with a given value of sales (or production) rate $y_0$ that the firm wants to improve (increase)
in order to generate a profit. For this reason, the next values $y_k$, $k > 0$, are assumed
to be greater than the initial value $y_0$. This economic assumption, makes valid
the Euler--Lagrange equation \eqref{eq:EL:P:after7.17}.
Indeed, it is known a priori, from economic insight,
that $y(t)$ is an increasing function \cite{CastilloPedregal}.
Similarly, the discretization from $(P)$ into
$(P_{\nabla\Delta})$ gives the discretized Euler--Lagrange equation
\begin{equation}
\label{E-L nabla delta:g}
a_{\Delta}(T)\cdot\left(\frac{\partial f_{1\nabla}}{\partial y_{k-1}}
-\nabla\frac{\partial f_{1\nabla}}{\partial \nabla y_{k}}\right)
+K_{\nabla}(T)\cdot\left(\frac{\partial f_{2\Delta}}{\partial y_{k+1}}
-\Delta\frac{\partial f_{2\Delta}}{\partial \Delta y_{k}}\right)=0
\end{equation}
that, for our example, reads
\begin{multline}
\label{eq:EL:P:after7.18}
\tag{$(EL_P)_{\nabla\Delta}$}
a_{\Delta}(T)(1-\rho)^{T-t_{k}}\left[
c_{1}-p_{0}+\frac{By_{0}}{(y_{k-1}-y_{0})^{2}}
-2c_{2}\left(\rho\nabla y_{k}+(1-\rho)\nabla^{2} y_{k}\right)
\right]\\
+ K_{\nabla}(T)(1+\rho)^{t_{k}-T}\left[
\lambda - \frac{
\beta \left(
\rho\sqrt{\Delta {y}_{k}+b}-\Delta\sqrt{\Delta {y}_{k}+b}\right)}
{2\sqrt{\Delta {y}_{k}+b}\sqrt{\Delta {y}_{k+1}+b}}\right]  = 0,
\end{multline}
$t_{k}\in\mathbb{T}_{\kappa}^{\kappa}$;
the discretization from $(P)$ into $(P_{\Delta\Delta})$
leads to the discretized Euler--Lagrange equation
\begin{equation}
\label{E-L delta delta:g}
a_{\Delta}(T)\cdot\left(\frac{\partial f_{1\Delta}}{\partial y_{k+1}}
-\Delta\frac{\partial f_{1\Delta}}{\partial \Delta y_{k}}\right)
+K_{\Delta}(T)\cdot\left(\frac{\partial f_{2\Delta}}{\partial y_{k+1}}
-\Delta\frac{\partial f_{2\Delta}}{\partial \Delta y_{k}}\right)=0
\end{equation}
and to
\begin{multline}
\label{eq:EL:P:after7.19}
\tag{$(EL_P)_{\Delta\Delta}$}
a_{\Delta}(T)(1+\rho)^{t_{k}-T}\left[
c_{1}-p_{0}+\frac{By_{0}}{(y_{k+1}-y_{0})^{2}}
-2c_{2}\left(\rho\Delta y_{k}+(1+\rho)\Delta^{2} y_{k}\right)
\right]\\
+K_{\Delta}(T)(1+\rho)^{t_{k}-T}\left[
\lambda - \frac{
\beta \left(
\rho\sqrt{\Delta {y}_{k}+b}-\Delta\sqrt{\Delta {y}_{k}+b}\right)}
{2\sqrt{\Delta {y}_{k}+b}\sqrt{\Delta {y}_{k+1}+b}}\right] = 0,
\end{multline}
$t_{k}\in\mathbb{T}^{\kappa^2}$; while the discretization from $(P)$ into
problem $(P_{\nabla\nabla})$ gives
\begin{equation}
\label{E-L nabla nabla:g}
a_{\nabla}(T)\cdot\left(\frac{\partial f_{1\nabla}}{\partial y_{k-1}}
-\nabla\frac{\partial f_{1\nabla}}{\partial \nabla y_{k}}\right)
+K_{\nabla}(T)\cdot\left(\frac{\partial f_{2\nabla}}{\partial y_{k-1}}
-\nabla\frac{\partial f_{2\nabla}}{\partial \nabla y_{k}}\right)=0
\end{equation}
that reduces in our case to
\begin{multline}
\label{eq:EL:P:after7.20}
\tag{$(EL_P)_{\nabla\nabla}$}
a_{\nabla}(T)(1-\rho)^{T-t_{k}}\left[
c_{1}-p_{0}+\frac{By_{0}}{(y_{k-1}-y_{0})^{2}}
-2c_{2}\left(\rho\nabla y_{k}+(1-\rho)\nabla^{2} y_{k}\right)\right]\\
+ K_{\nabla}(T)(1-\rho)^{T-t_{k}}\left[
\lambda  - \frac{
\beta \left(
\rho\sqrt{\nabla {y}_{k}+b}-\nabla\sqrt{\nabla {y}_{k}+b}\right)}
{2\sqrt{\nabla {y}_{k}+b}\sqrt{\nabla {y}_{k-1}+b}}\right] = 0,
\end{multline}
valid for $t_{k}\in\mathbb{T}_{\kappa^2}$.
As can be easily noticed, all the four discretizations of
the continuous Euler--Lagrange equation \eqref{E-L continuous} are different
but consist of the same items. For this reason, we define:
$$
\gamma_{1\Delta} := \left(\frac{\partial f_{1\Delta}}{\partial y_{k+1}}
-\Delta\frac{\partial f_{1\Delta}}{\partial \Delta y_{k}}\right),
\quad
\gamma_{1\nabla} := \left(\frac{\partial f_{1\nabla}}{\partial y_{k-1}}
-\nabla\frac{\partial f_{1\nabla}}{\partial \nabla y_{k}}\right),
$$
$$
\gamma_{2\Delta} := \left(\frac{\partial f_{2\Delta}}{\partial y_{k+1}}
-\Delta\frac{\partial f_{2\Delta}}{\partial \Delta y_{k}}\right),
\quad
\gamma_{2\nabla} := \left(\frac{\partial f_{2\nabla}}{\partial y_{k-1}}
-\nabla\frac{\partial f_{2\nabla}}{\partial \nabla y_{k}}\right).
$$
With such notations, the discretizations of the Euler--Lagrange equation
\eqref{E-L continuous} are conveniently written in the following way:
\begin{enumerate}
\item equation \eqref{E-L delta nabla:g} for $(P_{\Delta\nabla})$ is written as
\begin{equation}
\label{eq:EL:DN}
a_{\nabla}(T)\gamma_{1\Delta}+K_{\Delta}(T)\gamma_{2\nabla}=0,
\quad t_{k}\in\mathbb{T}_{\kappa}^{\kappa};
\end{equation}
\item equation \eqref{E-L nabla delta:g} for $(P_{\nabla\Delta})$ is written as
\begin{equation}
\label{eq:EL:ND}
a_{\Delta}(T)\gamma_{1\nabla}+K_{\nabla}(T)\gamma_{2\Delta}=0,
\quad t_{k}\in\mathbb{T}_{\kappa}^{\kappa};
\end{equation}
\item equation \eqref{E-L delta delta:g} for $(P_{\Delta\Delta})$ is written as
\begin{equation}
\label{eq:EL:DD}
a_{\Delta}(T)\gamma_{1\Delta}+K_{\Delta}(T)\gamma_{2\Delta}=0,
\quad t_{k}\in\mathbb{T}^{\kappa^2};
\end{equation}
\item and equation \eqref{E-L nabla nabla:g} for $(P_{\nabla\nabla})$
is equivalently written as
\begin{equation}
\label{eq:EL:NN}
a_{\nabla}(T)\gamma_{1\nabla}+K_{\nabla}(T)\gamma_{2\nabla}=0,
\quad t_{k}\in\mathbb{T}_{\kappa^2}.
\end{equation}
\end{enumerate}
\end{ex}


\subsection{Time-scale Euler--Lagrange equations}
\label{sec:AP:2}

The equation \eqref{eq:EL:DD} for problem $(P_{\Delta\Delta})$
coincides with the time-scale Euler--Lagrange delta equation given by
\cite[Corollary~3.4]{MalinowskaTorresCompositionDelta}
while equation \eqref{eq:EL:NN} for problem $(P_{\nabla\nabla})$
coincides with the time-scale Euler--Lagrange equation
given by \cite[Corollary~3.4]{MalinowskaTorresCompositionNabla}.
From our Corollary~\ref{cor in Z} it follows that such coincidence,
between the direct discretization of the continuous Euler--Lagrange
equation \eqref{E-L continuous} and the discrete Euler--Lagrange equations
\eqref{eq 1 delta Z}--\eqref{eq 2 nabla Z}
obtained from the calculus of variations on time scales,
does not hold for mixed delta-nabla discretizations: neither \eqref{eq:EL:DN}
is a time-scale Euler--Lagrange equation \eqref{eq 1 delta Z} or \eqref{eq 2 nabla Z}
for $(P_{\Delta\nabla})$ nor \eqref{eq:EL:ND} is a time-scale Euler--Lagrange equation
\eqref{eq 1 delta Z} or \eqref{eq 2 nabla Z} for $(P_{\nabla\Delta})$.

For the economic problem $(P_{\Delta\nabla})$ the Euler--Lagrange equations have the following form:
the Euler--Lagrange equation \eqref{eq 1 delta Z} takes the form
\begin{equation}
\label{eq 1 delta ex}
\tag{$EL^1_{P_{\Delta\nabla}}$}
\begin{split}
&
a_{\nabla}(T)\left(1+\rho\right)^{t_{k}-T}
\left[
c_{1}-p_{0}+\frac{By_{0}}{(y_{k+1}-y_{0})^{2}}
-2c_{2}\left(\rho\Delta y_{k}+(1+\rho)\Delta^{2} y_{k}
\right)\right]\\
&
+ K_{\Delta}(T)(1-\rho)^{T-t_{k}}
\left( \lambda  - \frac{
\beta \left(\rho\sqrt{\nabla {y}_{k}+b}
-(1-\rho)\Delta\sqrt{\nabla {y}_{k}+b}\right)}
{2\sqrt{\nabla {y}_{k}+b}\sqrt{\nabla {y}_{k+1}+b}}\right)\\
&
+\Delta\left[
K_{\Delta}(T)(1-\rho)^{T-t_{k}}
\left(
\lambda  - \frac{\beta \left(
\rho\sqrt{\nabla {y}_{k}+b}-\nabla\sqrt{\nabla {y}_{k}+b}
\right)}{2\sqrt{\nabla {y}_{k}+b}\sqrt{\nabla {y}_{k-1}+b}}\right)\right](t_{k+1}) = 0
\end{split}
\end{equation}
for $t_{k}\in\mathbb{T}^{\kappa}_{\kappa}$,
while the Euler--Lagrange equation \eqref{eq 2 nabla Z} gives
\begin{equation}
\label{eq 2 nabla ex}
\tag{$EL^2_{P_{\Delta\nabla}}$}
\begin{split}
&
a_{\nabla}(T)\left(1+\rho\right)^{t_{k-1}-T}
\left[
c_{1}-p_{0}+\frac{By_{0}}{(y_{k}-y_{0})^{2}}
-2c_{2}\left( \rho\Delta y_{k} + \nabla\left(\Delta y_{k}\right)
\right)\right]\\
&
+ K_{\Delta}(T)(1-\rho)^{T-t_{k}}
\left(
\lambda
- \frac{
\beta \left(\rho\sqrt{\nabla {y}_{k1}+b}-\nabla
\sqrt{\nabla {y}_{k}+b}\right)}
{2\sqrt{\nabla {y}_{k}+b}\sqrt{\nabla {y}_{k-1}+b}}\right)\\
&
-\nabla\left[
a_{\nabla}(T)(1+\rho)^{t_{k}-T}
\left(c_{1}-p_{0}+\frac{By_{0}}{(y_{k}-y_{0})^{2}}-2c_{2}
\left(\rho\Delta y_{k}+(1+\rho)\Delta^{2}y_{k}\right)
\right)\right](t_{k-1})  = 0
\end{split}
\end{equation}
for $t_{k}\in\mathbb{T}_{\kappa}^{\kappa}$.
For problem $(P_{\nabla\Delta})$
the Euler--Lagrange equations take the following form:
the Euler--Lagrange equation \eqref{eq 1 delta Z} gives
\begin{equation}
\label{eq 1 delta ex nabladelta}
\tag{$EL^1_{P_{\nabla\Delta}}$}
\begin{split}
&
a_{\Delta}(T)\left(1-\rho\right)^{T-t_{k}-1}
\left[
c_{1}-p_{0}+\frac{By_{0}}{(y_{k}-y_{0})^{2}}
-2c_{2}\left(\rho\nabla y_{k}+\Delta(\nabla y_{k})
\right)
\right]\\
&
+ K_{\nabla}(T)(1+\rho)^{t_{k}-T}
\left[
\lambda - \frac{
\beta \left(\rho\sqrt{\Delta {y}_{k}+b}
-\Delta\sqrt{\Delta {y}_{k}+b}\right)}
{2\sqrt{\Delta {y}_{k}+b}\sqrt{\Delta {y}_{k+1}+b}}\right]\\
&
+ \Delta\left[
a_{\Delta}(T)(1-\rho)^{T-t_{k}}
\left[
c_{1}-p_{0}+\frac{By_{0}}{(y_{k-1}-y_{0})^{2}}
-2c_{2}\left(\rho\nabla y_{k}+(1-\rho)\nabla^2 y_{k}\right)
\right]
\right](t_{k+1}) = 0
\end{split}
\end{equation}
for $t_{k}\in\mathbb{T}^{\kappa}_{\kappa}$, and \eqref{eq 2 nabla Z} gives
\begin{equation}
\label{eq 2 nabla ex nabladelta}
\tag{$EL^2_{P_{\nabla\Delta}}$}
\begin{split}
&
a_{\Delta}(T)\left(1-\rho\right)^{T-t_{k}}
\left[
c_{1}-p_{0}+\frac{By_{0}}{(y_{k-1}-y_{0})^{2}}
-2c_{2}\left( \rho \nabla y_{k}+(1-\rho)\nabla^{2} y_{k}
\right)
\right]\\
&
+ K_{\nabla}(T)(1+\rho)^{t_{k-1}-T}
\left[
\lambda - \frac{
\beta \left(\rho\sqrt{\Delta {y}_{k}+b}
-(1+\rho)\nabla\sqrt{\Delta {y}_{k}+b}\right)}
{2\sqrt{\Delta {y}_{k}+b}\sqrt{\nabla{y}_{k}+b}}\right]\\
&
-\nabla\left[
K_{\nabla}(T)(1+\rho)^{t_{k}-T}
\left[
\lambda - \frac{\beta \left(
\rho\sqrt{\Delta {y}_{k}+b}-\Delta\sqrt{\Delta {y}_{k}+b}
\right)}
{2\sqrt{\Delta {y}_{k}+b}\sqrt{\Delta {y}_{k+1}+b}}\right]
\right](t_{k-1}) = 0
\end{split}
\end{equation}
for $t_{k}\in\mathbb{T}_{\kappa}^{\kappa}$.
Then the Euler--Lagrange equations \eqref{eq 1 delta ex}
and \eqref{eq 2 nabla ex} for $(P_{\Delta\nabla})$ are
\begin{equation*}
a_{\nabla}(T)\gamma_{1\Delta}+K_{\Delta}(T)\left(\frac{\partial f_{2\nabla}}{\partial  {y}_{k-1}}\circ \sigma
-\Delta\frac{\partial f_{2\nabla}}{\partial \nabla {y}_{k}}\right)+\Delta\left[K_{\Delta}(T)\gamma_{2\nabla}\right]\circ \sigma=0,
\quad t_{k}\in\mathbb{T}_{\kappa}^{\kappa},
\end{equation*}
and
\begin{equation*}
a_{\nabla}(T)\left(\frac{\partial f_{1\Delta}}{\partial  {y}_{k+1}}\circ \rho
-\nabla\frac{\partial f_{1\Delta}}{\partial \Delta {y}_{k}}\right)
+K_{\Delta}(T)\gamma_{2\nabla}-\nabla\left[a_{\nabla}(T)\gamma_{1\Delta}\right]\circ\rho=0,
\quad t_{k}\in\mathbb{T}_{\kappa}^{\kappa},
\end{equation*}
respectively, and the Euler--Lagrange equations \eqref{eq 1 delta ex nabladelta}
and \eqref{eq 2 nabla ex nabladelta} for $(P_{\nabla\Delta})$ are
\begin{equation*}
a_{\Delta}(T)\left(\frac{\partial f_{1\nabla}}{\partial{y}_{k-1}}\circ \sigma
-\Delta\frac{\partial f_{1\nabla}}{\partial \nabla {y}_{k}}\right)
+K_{\nabla}(T) \gamma_{2\Delta}
+\Delta\left[a_{\Delta}(T)\gamma_{1\nabla}\right]\circ \sigma=0,
\quad t_{k}\in\mathbb{T}_{\kappa}^{\kappa},
\end{equation*}
and
\begin{equation*}
a_{\Delta}(T)\gamma_{1\nabla}
+K_{\nabla}(T)\left(\frac{\partial f_{2\Delta}}{\partial{y}_{k+1}}\circ \rho
-\nabla\frac{\partial f_{2\Delta}}{\partial \Delta {y}_{k}}\right)
-\nabla\left[K_{\nabla}(T)\gamma_{2\Delta}\right]\circ\rho=0,
\quad t_{k}\in\mathbb{T}_{\kappa}^{\kappa},
\end{equation*}
respectively.

For the convenience of the reader, we recall the introduced notations:
\begin{itemize}
\item $P$ -- the continuous economic problem describing
a market policy of a firm, presented in Section~\ref{sec:ApplEconomics};

\item $EL_{P}$ -- the continuous Euler--Lagrange
equation \eqref{E-L continuous} associated to problem $P$ (see \eqref{eq:contELeq});

\item $P_{D}$ -- a discretization of problem $P$, in four possible forms:
$D\in\left\{ \Delta\Delta, \nabla\nabla, \Delta\nabla,\nabla\Delta \right\}$;

\item $(EL_{P})_{D}$ -- a discretization of the Euler--Lagrange equation $EL_{P}$,
in four different forms:
$D\in\left\{ \Delta\Delta, \nabla\nabla, \Delta\nabla,\nabla\Delta \right\}$;

\item $EL_{P_{D}}$ -- discrete Euler--Lagrange equations associated to problem $P_{D}$,
obtained from the calculus of variations on time scales (see Corollary~\ref{cor in Z}).
\end{itemize}


\section{Standard versus time-scale discretizations}
\label{sec:compare}

The discrepancy between direct discretization of the classical
optimality conditions and the time-scale approach to the calculus of variations
was discussed, from an embedding point of view, in \cite{MR2966865}.
Here we compare the results obtained
from direct and time-scale discretizations
for the more general problem \eqref{problem}--\eqref{punkty},
in concrete for the economic problem $(P)$
discussed in Section~\ref{sec:ApplEconomics}.
For illustrative purposes, the following values have been borrowed from \cite{CastilloPedregal}:
$$
\rho=0.05,\quad c_{0}=3,\quad c_{1}=0.5,\quad c_{2}=3,\quad T=3,
$$
$$
b=4,\quad \lambda= \frac{1}{2},\quad \beta=\frac{1}{4},\quad B=2,\quad y_{0}=2,
\quad y_T = 3.
$$
Moreover, we fixed the time scale to be $\mathbb{T}=\left\{0,1,2,3\right\}$.
In what follows we compare the candidates for solutions of the variational problems
$(P_{\Delta\nabla})$, $(P_{\nabla\Delta})$, $(P_{\Delta\Delta})$, and $(P_{\nabla\nabla})$,
obtained from the direct discretizations of the continuous Euler--Lagrange equation
(Section~\ref{sec:AP:1}) and the discrete time-scale
Euler--Lagrange equations (Section~\ref{sec:AP:2}).
All calculations were done using the Computer Algebra System \textsf{Maple}, version~10
(see Appendix~\ref{App:MapleCode}).
For problems $(P_{\Delta\Delta})$ and $(P_{\nabla\nabla})$
the discretization of the continuous Euler--Lagrange equation
and the discrete time-scale Euler--Lagrange equations coincide.
The Euler--Lagrange equation for problem $(P_{\Delta\Delta})$ is defined on
$\mathbb{T}^{\kappa^{2}}=\left\{0,1\right\}$ and we obtain a system
of two equations with two unknowns $y_{1}$ and $y_{2}$ that leads to
$y_{1}=2.322251304$ and $y_{2}=2.679109437$
with the cost functional value $K_{\Delta}(T)a_{\Delta}(T) = -16.97843026$.
Similarly, the Euler--Lagrange equation for problem $(P_{\nabla\nabla})$ is defined on
$\mathbb{T}_{\kappa^{2}}=\left\{2,3\right\}$ and we obtain a system
of two equations with two unknowns $y_{1}$ and $y_{2}$ that leads to
$y_{1}=1.495415602$ and $y_{2}=2.228040364$
with the cost functional value
$K_{\nabla}(T)a_{\nabla}(T) = -13.20842214$. As we show next,
for hybrid delta-nabla discrete problems of the calculus of variations,
the time-scale results seem superior.


\subsection{Problem $(P_{\Delta\nabla})$}

The Euler--Lagrange equations for problem $(P_{\Delta\nabla})$ are defined on
$\mathbb{T}^{\kappa}_{\kappa}=\left\{1,2\right\}$. Therefore, we obtain a system of equations
with two unknowns $y_{1}$ and $y_{2}$.
The discretized Euler--Lagrange equation \eqref{eq:EL:P:after7.17} gives
\begin{equation*}
y_{1}=2.910488556, \quad y_{2} = 2.970017180
\end{equation*}
with value of cost functional
$$
K_{\Delta}(T)a_{\nabla}(T) = -10.11399047.
$$
A better result is obtained using the discrete time-scale
Euler--Lagrange equation \eqref{eq 1 delta ex}:
$$
y_{1}=2.901851949,\quad y_{2} = 2.967442285
$$
with cost
$$
K_{\Delta}(T)a_{\nabla}(T) = -10.30544712.
$$


\subsection{Problem $(P_{\nabla\Delta})$}

The Euler--Lagrange equations for problem $(P_{\nabla\Delta})$ are also defined on
$\mathbb{T}^{\kappa}_{\kappa}=\left\{1,2\right\}$ and also lead
to a system of two equations with the two unknowns $y_{1}$ and $y_{2}$.
The discretized Euler--Lagrange equation \eqref{eq:EL:P:after7.18} gives
$$
y_{1}=2.183517532, \quad y_{2}=2.446990272
$$
with cost
$$
K_{\nabla}(T)a_{\Delta}(T) = -19.09167089.
$$
Our time-scale Euler--Lagrange equation
\eqref{eq 2 nabla ex nabladelta} gives better results:
$$
y_{1}=2.186742579,\quad y_{2}=2.457402400
$$
with cost
$$
K_{\nabla}(T)a_{\Delta}(T) = -19.17699675.
$$
The results are gathered in Table~\ref{Tbl:Ch7:tbl:1}.
\begin{table}[ht]
\begin{center}
\begin{tabular}{|c|c|c|c|}\hline
\multirow{2}{*}{$D$} &
\multicolumn{3}{c|}{The value of the functional of $(P_D)$, $\rho =0.05$, for candidates to minimizers obtained from:}\\
\cline{2-4}
&$(EL_{P})_{D}$  &$EL^{1}_{P_{D}}$  & $EL^{2}_{P_{D}}$ \\ \hline
$\Delta\nabla$ & $ -10.11399047$& $ -10.30544712 $ & $ -0.1537986252 \times 10^{-5}$   \\ \hline
$\nabla\Delta$  & $ -19.09167089 $ & $ 1020.105142 $ & $-19.17699675 $ \\ \hline
$\Delta\Delta$ & \multicolumn{3}{c|}{ -16.97843026 } \\ \hline
$\nabla\nabla$  & \multicolumn{3}{c|}{ -13.20842214}  \\ \hline
\end{tabular}
\end{center}
\caption{\label{Tbl:Ch7:tbl:1}The value of the functional associated to problem $P_{D}$,
$D \in \{\Delta\nabla, \nabla\Delta, \Delta\Delta, \nabla\nabla\}$, with $\rho =0.05$,
calculated using: (i) the direct discretization of the continuous Euler--Lagrange equation, that is, $(EL_{P})_{D}$;
(ii) discrete Euler--Lagrange equations $EL_{P_{D}}$,
obtained from the calculus of variations on time scales with $\mathbb{T} = \mathbb{Z}$.}
\end{table}


\section{Conclusion}
\label{sec:conc}

Some advantages of using the calculus of variations on time scales
in economics were already discussed in \cite{MR2218315,Atici:comparison,MyID:267}.
Here we considered two minimization discrete delta-nabla economic problems,
denoted by $(P_{\Delta\nabla})$ and $(P_{\nabla\Delta})$,
for which the time-scale approach leads to better results than the ones
obtained by a direct discretization of the continuous necessary optimality condition:
the approach on the right hand side of the diagram of Figure~\ref{diagram}
gives candidates to minimizers for which
the value of the functional is smaller than the values obtained from
the approach on the left hand side of the diagram of Figure~\ref{diagram}.
It might be concluded that the time-scale theory of the calculus of variations
leads to more precise results than the standard methods of discretization.
For comparison purposes, we have used the same values for the parameters
as the ones available in \cite{CastilloPedregal}. We have, however,
done simulations with other values of the parameters and the conclusion
persists: in almost all cases the results obtained from our time-scale approach are better;
hardly ever, they coincide with the classical method; never are worse.
In particular, we changed the value of the discount rate, $\rho$,
in the set $\{0.01, 0.02, 0.03, \ldots, 0.1\}$. This is motivated by the fact that this value
depends much on the economic and politic situation. The case where the time-scale advantage
is more visible is given in Table~\ref{Tbl:Ch7:tbl:2}, which corresponds to
a discount rate of 2\% ($\rho =0.02$). The interested reader can easily do his/her own
simulations using the \textsf{Maple} code found in Appendix~\ref{App:MapleCode}.
For future work, we would like to generalize our mixed delta-nabla results,
in particular Theorem~\ref{main}, for infinite horizon variational problems on time scales,
that so far have been only studied in the delta \cite{MR2747272} and nabla \cite{MyID:254} cases.
\begin{table}[ht]
\begin{center}
\begin{tabular}{|c|c|c|c|}\hline
\multirow{2}{*}{$D$} &
\multicolumn{3}{c|}{The value of the functional of $(P_D)$, $\rho =0.02$, for candidates to minimizers obtained from:}\\
\cline{2-4}
&$(EL_{P})_{D}$  &$EL^{1}_{P_{D}}$  & $EL^{2}_{P_{D}}$ \\ \hline
$\Delta\nabla$ & $ -10.62044023 $& $ -10.70908681 $ & $ 0.00001078869584 $   \\ \hline
$\nabla\Delta$  & $ -21.05128963 $ & $ 3.014255571 \times 10^{-8} $ & $ -264.5250742 $ \\ \hline
$\Delta\Delta$ & \multicolumn{3}{c|}{ -19.03571446 } \\ \hline
$\nabla\nabla$  & \multicolumn{3}{c|}{ -14.19294557}  \\ \hline
\end{tabular}
\end{center}
\caption{\label{Tbl:Ch7:tbl:2}The value of the functional associated to problem $P_{D}$,
$D \in \{\Delta\nabla, \nabla\Delta, \Delta\Delta, \nabla\nabla\}$, with $\rho =0.02$,
calculated using: (i) the direct discretization of the continuous Euler--Lagrange equation, that is, $(EL_{P})_{D}$;
(ii) discrete Euler--Lagrange equations $EL_{P_{D}}$,
obtained from the calculus of variations on time scales with $\mathbb{T} = \mathbb{Z}$.}
\end{table}


\appendix


\section{\textsf{Maple} Code}
\label{App:MapleCode}

We provide here all the definitions and computations done in \textsf{Maple} for
the problems considered in Section~\ref{sec:compare}. The definitions
follow closely the notations introduced along the paper, and should be clear
even for readers not familiar with the Computer Algebra System \textsf{Maple}.
\small
\begin{verbatim}
> restart:
> rho := 5/100:
> c0 := 3:
> lambda := 1/2:
> c1 := 1/2:
> c2 := 3:
> p0 := 1:
> y0 := 1:
> b := 4:
> beta := 1/4:
> B := 2:
> T := 3:
> y(0) := 2:
> y(T) := 3:
> TimeScale := [seq(i,i=0..T)];
\end{verbatim}
$$
TimeScale := [0, 1, 2, 3]
$$
\begin{verbatim}
> Sigma := t-> piecewise(t < T, t+1, t):
> Rho := t -> piecewise(t > 0, t-1, t):
> Delta := f -> f@Sigma-f:
> Nabla := f -> f-f@Rho:
> KDelta := sum((1+rho)^(t-T)*(c0+c1*(y@Sigma)(t)+c2*(Delta(y)(t))^2
     -(y@Sigma)(t)*p0-(B*(y@Sigma)(t))/((y@Sigma)(t)-y0)),t=0..T-1):
> KNabla := sum((1-rho)^(T-t)*(c0+c1*(y@Rho)(t)+c2*(Nabla(y)(t))^2
     -(y@Rho)(t)*p0-(B*(y@Rho)(t))/((y@Rho)(t)-y0)),t=1..T):
> aDelta := sum((1+rho)^(t-T)*(lambda*(y@Sigma)(t)
     +beta*sqrt(Delta(y)(t)+b)),t=0..T-1):
> aNabla := sum((1-rho)^(T-t)*(lambda*(y@Rho)(t)
     +beta*sqrt(Nabla(y)(t)+b)),t=1..T):
> Functional_PDN := subs({y(1)=y1,y(2)=y2},KDelta*aNabla):
> Functional_PND := subs({y(1)=y1,y(2)=y2},KNabla*aDelta):
> Functional_PDD := subs({y(1)=y1,y(2)=y2},KDelta*aDelta):
> Functional_PNN := subs({y(1)=y1,y(2)=y2},KNabla*aNabla):
> gamma1delta := t -> (1+rho)^(t-T)*(((c1-p0+(B*y0)/(((y@Sigma)(t)-y0)^2)))
     -2*c2*(rho*Delta(y)(t)+(1+rho)*Delta(Delta(y))(t))):
> gamma1nabla := t -> (1-rho)^(T-t)*((c1-p0+(B*y0)/(((y@Rho)(t)-y0)^2))
     -2*c2*(rho*Nabla(y)(t)+(1-rho)*Nabla(Nabla(y))(t))):
> gamma2delta := t -> (1+rho)^(t-T)*(lambda-(beta*(rho*sqrt(Delta(y)(t)+b)
     -(Delta(unapply(sqrt(Delta(y)(s)+b),s))(t))))/(2*sqrt(Delta(y)(t)+b)
     *sqrt((Delta(y)@Sigma)(t)+b))):
> gamma2nabla := t -> (1-rho)^(T-t)*(lambda
     -(beta*(rho*sqrt(Nabla(y)(t)+b)-Nabla(unapply(sqrt(Nabla(y)(s)+b),s))(t)))
     /(2*sqrt(Nabla(y)(t)+b)*sqrt((Nabla(y)@Rho)(t)+b))):
> # now we define the 4 problems that are considered in the paper
> # discretization of the continuous E-L equations
> # Problem Delta Nabla PDN
> # domain T_{kappa}^{kappa}
> PDN := t -> aNabla*gamma1delta(t)+KDelta*gamma2nabla(t):
> # Problem  Nabla Delta PND
> # domain T_{kappa}^{kappa}
> PND := t -> aDelta*gamma1nabla(t)+KNabla*gamma2delta(t):
> # Problem Delta Delta  PDD
> # domain T^{kappa^2}
> PDD := t -> aDelta*gamma1delta(t)+KDelta*gamma2delta(t):
> # Problem  Nabla Nabla PNN
> # domain T_{kappa^2}
> PNN := t -> aNabla*gamma1nabla(t)+KNabla*gamma2nabla(t):
> eqPDN := subs({y(1)=y1,y(2)=y2},{PDN(1)=0,PDN(2)=0}):
> SolutionPDN := fsolve(eqPDN,{y1,y2});
\end{verbatim}
$$
SolutionPDN := \{y1 = 2.910488556, y2 = 2.970017180\}
$$
\begin{verbatim}
> subs(SolutionPDN,Functional_PDN);
\end{verbatim}
$$
-10.11399047
$$
\begin{verbatim}
> eqPND := subs({y(1)=y1,y(2)=y2},{PND(1)=0,PND(2)=0}):
> SolutionPND := fsolve(eqPND,{y1,y2});
\end{verbatim}
$$
SolutionPND := \{y1 = 2.183517532, y2 = 2.446990272\}
$$
\begin{verbatim}
subs(SolutionPND,Functional_PND);
\end{verbatim}
$$
-19.09167089
$$
\begin{verbatim}
> eqPDD := subs({y(1)=y1,y(2)=y2},{PDD(0)=0,PDD(1)=0}):
> SolutionPDD := fsolve(eqPDD,{y1,y2});
\end{verbatim}
$$
SolutionPDD := \{y1 = 2.322251304, y2 = 2.679109437\}
$$
\begin{verbatim}
> subs(SolutionPDD,Functional_PDD);
\end{verbatim}
$$
-16.97843026
$$
\begin{verbatim}
> eqPNN := subs({y(1)=y1,y(2)=y2},{PNN(2)=0,PNN(3)=0}):
> SolutionPNN := fsolve(eqPNN,{y1,y2});
\end{verbatim}
$$
SolutionPNN := \{y1 = 1.495415602, y2 = 2.228040364\}
$$
\begin{verbatim}
> subs(SolutionPNN,Functional_PNN);
\end{verbatim}
$$
-13.20842214
$$
\begin{verbatim}
> # discretization of the time scale Euler-Lagrange equations
> # domain T_{kappa}^{kappa}
> part1 := t -> lambda*(1-rho)^(T-Sigma(t)):
> part2 := t ->(beta*(1-rho)^(T-Sigma(t))*((rho*sqrt(Nabla(y)(t)+b))
     -(1-rho)*(Delta(unapply(sqrt(Nabla(y)(s)+b),s))(t))))
     /(2*sqrt(Nabla(y)(t)+b)*sqrt(Delta(y)(t)+b)):
> part3 := t -> (1+rho)^(Rho(t)-T)*(c1-p0+(B*y0)/((y(t)-y0)^2)):
> part4 := t -> 2*c2*(1+rho)^(Rho(t)-T)
     *(rho*Delta(y)(t)+(y@Sigma)(t)-2*y(t)+(y@Rho)(t)):
> partDelta := Delta(unapply(KDelta*gamma2nabla(t),t))@Sigma:
> partNabla := Nabla(unapply(aNabla*gamma1delta(t),t))@Rho:
\end{verbatim}
\verb!> # E-L equation !\eqref{eq 1 delta Z}\verb! for Problem Delta Nabla!
\begin{verbatim}
> EL_delta := t -> aNabla*gamma1delta(t)+KDelta*(part1(t)-part2(t))+partDelta(t):
\end{verbatim}
\verb!> # E-L equation !\eqref{eq 2 nabla Z}\verb! for Problem Delta Nabla!
\begin{verbatim}
> EL_nabla := t -> aNabla*(part3(t)-part4(t))+KDelta*gamma2nabla(t)-partNabla(t):
> # systems of E-L equations for Problem Delta Nabla
> EL_delta_system := subs({y(1)=y1,y(2)=y2},{EL_delta(1)=0,EL_delta(2)=0}):
> Solution_EL_eqs_system_delta_version := fsolve(EL_delta_system,{y1,y2});
\end{verbatim}
$$
Solution_EL_eqs_system_delta_version := \{y1 = 2.901851949, y2 = 2.967442285\}
$$
\begin{verbatim}
> subs(Solution_EL_eqs_system_delta_version,Functional_PDN);
\end{verbatim}
$$
-10.30544712
$$
\begin{verbatim}
> EL_nabla_system := subs({y(1)=y1,y(2)=y2},{EL_nabla(1)=0,EL_nabla(2)=0}):
> Solution_EL_eqs_system_nabla_version := fsolve(EL_nabla_system,{y1,y2});
\end{verbatim}
$$
\{y1 = 0.5930298703, y2 = 1.090438395\}
$$
\begin{verbatim}
subs(Solution_EL_eqs_system_nabla_version,Functional_PDN);
\end{verbatim}
$$
-0.000001537986252
$$
\begin{verbatim}
> # E-L equations for Problem Nabla Delta
> part5 := t -> (1-rho)^(T-Sigma(t))*(c1-p0+(B*y0)/((y(t)-y0)^2)):
> part6 := t -> 2*c2*(1-rho)^(T-Sigma(t))*(rho*(Nabla(y)(t))+(Delta(Nabla(y))(t))):
> part7 := t -> lambda*(1+rho)^(Rho(t)-T):
> part8 := t -> (1+rho)^(Rho(t)-T)*((beta*(rho*sqrt(Delta(y)(t)+b)
     -(1+rho)*Nabla(unapply(sqrt(Delta(y)(s)+b),s))(t)))
     /(2*sqrt(Delta(y)(t)+b)*sqrt(Nabla(y)(t)+b))):
> partDelta2 := Delta(unapply(aDelta*gamma1nabla(t),t))@Sigma:
> partNabla2 := Nabla(unapply(KNabla*gamma2delta(t),t))@Rho:
\end{verbatim}
\verb!> # E-L equation !\eqref{eq 1 delta Z}\verb! for Problem Nabla Delta!
\begin{verbatim}
> EL_delta2 := t -> KNabla*gamma2delta(t)+aDelta*(part5(t)-part6(t))+partDelta2(t):
\end{verbatim}
\verb!> # E-L equation !\eqref{eq 2 nabla Z}\verb! for Problem Nabla Delta!
\begin{verbatim}
> EL_nabla2 := t -> KNabla*(part7(t)-part8(t))+aDelta*gamma1nabla(t)-partNabla2(t):
> # systems of E-L equations for Problem Nabla Delta
> EL_delta2_system := subs({y(1)=y1,y(2)=y2},{EL_delta2(1)=0,EL_delta2(2)=0}):
> Solution_EL_eqs_system_delta2_version := fsolve(EL_delta2_system,{y1,y2});
\end{verbatim}
$$
Solution_EL_eqs_system_delta2_version := \{y1 = 7.879260741, y2 = 4.775003718\}
$$
\begin{verbatim}
> subs(Solution_EL_eqs_system_delta2_version,Functional_PND);
\end{verbatim}
$$
1020.105142
$$
\begin{verbatim}
> EL_nabla2_system := subs({y(1)=y1,y(2)=y2},{EL_nabla2(1)=0,EL_nabla2(2)=0}):
> Solution_EL_eqs_system_nabla2_version := fsolve(EL_nabla2_system,{y1,y2});
\end{verbatim}
$$
Solution_EL_eqs_system_nabla2_version := \{y1 = 2.186742579, y2 = 2.457402400\}
$$
\begin{verbatim}
> subs(Solution_EL_eqs_system_nabla2_version,Functional_PND);
\end{verbatim}
$$
-19.17699675
$$

\normalsize


\section*{Acknowledgments}

This work was partially supported by Portuguese funds through the
\emph{Center for Research and Development in Mathematics and Applications} (CIDMA),
and \emph{The Portuguese Foundation for Science and Technology} (FCT),
within project PEst-OE/MAT/UI4106/2014.
Dryl was also supported by FCT through the Ph.D. fellowship
SFRH/BD/51163/2010; Torres by FCT within project OCHERA,
PTDC/EEI-AUT/1450/2012, co-financed by FEDER under POFC-QREN
with COMPETE reference FCOMP-01-0124-FEDER-028894.
The authors are very grateful to three anonymous referees,
for several constructive remarks and suggestions.



\end{document}